\documentclass[12pt,british,a4paper,reqno]{amsart}
\usepackage[T1]{fontenc}
\usepackage[latin9]{inputenc}
\usepackage{geometry}
\usepackage{textcomp}
\usepackage{mathrsfs}
\usepackage{mathtools}
\usepackage{amsmath}
\usepackage{amsthm}
\usepackage{amssymb}
\usepackage{stmaryrd}
\usepackage{setspace}
\makeatletter
\numberwithin{figure}{section}
\numberwithin{equation}{section}
\theoremstyle{plain}
\newtheorem{thm}{\protect\theoremname}[section]
  \theoremstyle{definition}
  \newtheorem{defn}[thm]{\protect\definitionname}
  \theoremstyle{plain}
  \newtheorem{prop}[thm]{\protect\propositionname}
  \theoremstyle{remark}
  \newtheorem{rem}[thm]{\protect\remarkname}
  \theoremstyle{plain}
  \newtheorem{lem}[thm]{\protect\lemmaname}
  \theoremstyle{plain}
  \newtheorem{cor}[thm]{\protect\corollaryname}
  \theoremstyle{definition}
  \newtheorem{example}[thm]{\protect\examplename}

\usepackage{amsmath}
\usepackage{amssymb}
\usepackage{amsfonts}
\usepackage{braket}
\usepackage{fontenc}
\usepackage{ifsym}
\usepackage{proof}
\usepackage{enumerate}
\usepackage{graphicx}
\usepackage{qtree}
\usepackage{tikz}
\usepackage{tree-dvips}
\usepackage{tikz-cd}
\usepackage{hyperref}
\usepackage{bookmark}
\usepackage{fancyhdr}
\usepackage{mathrsfs}
\usepackage{color}
\usepackage{indentfirst}
\usepackage{mathtools}
\usepackage{tensor}

\makeatletter
\usetikzlibrary{matrix,arrows,decorations.pathmorphing,positioning}
\tikzset{commutative diagrams/.cd, mysymbol/.style={start anchor=center,end anchor=center,draw=none}}
\makeatother

\newcommand{\commutes}[2][\circ]{\arrow[mysymbol]{#2}[description]{#1}}

\newcommand{\PB}[2][\square]{\arrow[mysymbol,pos=0.27]{#2}[description]{#1}}
\newcommand{\PO}[2][\square]{\arrow[mysymbol,pos=0.27]{#2}[description]{#1}}

\makeatletter
\newcommand{\dotminus}{\mathbin{\text{\@dotminus}}}
\newcommand{\@dotminus}{%
  \ooalign{\hidewidth\raise1ex\hbox{.}\hidewidth\cr$\m@th-$\cr}}
\makeatother

\makeatletter
\newcommand*{\rom}[1]{\expandafter\@slowromancap\romannumeral #1@}
\makeatother

\newcommand{\CA}{\mathcal{A}}

\newcommand{\CC}{\mathcal{C}}

\newcommand{\CH}{\mathcal{H}}
\newcommand{\CI}{\mathcal{I}}

\newcommand{\CP}{\mathcal{P}}

\newcommand{\CR}{\mathcal{R}}
\newcommand{\CS}{\mathcal{S}}
\newcommand{\CT}{\mathcal{T}}
\newcommand{\CU}{\mathcal{U}}
\newcommand{\CV}{\mathcal{V}}
\newcommand{\CW}{\mathcal{W}}
\newcommand{\CX}{\mathcal{X}}
\newcommand{\CY}{\mathcal{Y}}

\newcommand{\BN}{\mathbb{N}}

\newcommand{\deff}{\coloneqq}

\newcommand{\Hom}{\operatorname{Hom}\nolimits}
\newcommand{\obj}{\operatorname{obj}\nolimits}

\renewcommand{\leq}{\leqslant}
\renewcommand{\geq}{\geqslant}
\renewcommand{\phi}{\varphi}

\newcommand{\iso}{\cong}

\bookmarksetup{numbered,open}

\newsavebox{\wideeqbox}

\let\amph=&

\newcommand{\sus}{\Sigma}
\newcommand{\add}{\operatorname{\textsf{add}}\nolimits}
\newcommand{\End}{\operatorname{End}\nolimits}

\newcommand{\Ext}{\operatorname{Ext}\nolimits}

\renewcommand{\mod}{\operatorname{\textsf{mod}}\nolimits}

\newcommand{\XR}{\operatorname{\mathcal{X}_\emph{R}}\nolimits}
\newcommand{\XT}{\operatorname{\mathcal{X}_\emph{T}}\nolimits}

\newcommand{\op}{\text{op}}

\newcommand{\Ker}{\operatorname{Ker}\nolimits}

\newcommand{\cok}{\operatorname{coker}\nolimits}
\newcommand{\Image}{\operatorname{Im}\nolimits}
\newcommand{\Coim}{\operatorname{Coim}\nolimits}

\newcommand{\image}{\operatorname{im}\nolimits}
\newcommand{\coim}{\operatorname{coim}\nolimits}

\newcommand{\ol}[1]{\overline{#1}}
\newcommand{\LF}{\mathrm{LF}}

\newenvironment{acknowledgements}{\begin{abstract}} {\end{abstract}}

\makeatother

\usepackage{babel}
  \providecommand{\corollaryname}{Corollary}
  \providecommand{\definitionname}{Definition}
  \providecommand{\examplename}{Example}
  \providecommand{\lemmaname}{Lemma}
  \providecommand{\propositionname}{Proposition}
  \providecommand{\remarkname}{Remark}
\providecommand{\theoremname}{Theorem}

\begin{document}

\title{Quasi-abelian hearts of twin cotorsion pairs on triangulated categories\footnote{\url{https://doi.org/10.1016/j.jalgebra.2019.06.011}}}

\author{Amit Shah}
\address{School of Mathematics \\ 
University of Leeds \\ 
Leeds, LS2 9JT \\ 
United Kingdom
}
\email{mmas@leeds.ac.uk}
\date{\today}

\begin{abstract}
We prove that, under a mild assumption, the heart $\overline{\CH}$
of a twin cotorsion pair $((\CS,\CT),(\CU,\CV))$ on a triangulated
category $\CC$ is a quasi-abelian category. If $\CC$ is also Krull-Schmidt
and $\CT=\CU$, we show that the heart of the cotorsion pair $(\CS,\CT)$
is equivalent to the Gabriel-Zisman localisation of $\overline{\CH}$
at the class of its regular morphisms.

In particular, suppose $\CC$ is a cluster category with a rigid object
$R$ and $[\CX_{R}]$ the ideal of morphisms factoring through $\CX_{R}=\Ker(\Hom_{\CC}(R,-))$,
then applications of our results show that $\CC/[\CX_{R}]$ is a quasi-abelian
category. We also obtain a new proof of an equivalence between the
localisation of this category at its class of regular morphisms and
a certain subfactor category of $\CC$.
\end{abstract}
\maketitle

\section{Introduction\label{sec:Introduction}}

Cotorsion pairs were first defined specifically for the category of
abelian groups in \cite{Salce-cotorsion-theories-for-abelian-groups}
as an analogue of the torsion theories introduced in \cite{Dickson-torsion-theory-for-abelian-cats},
which were themselves used to generalise the notion of torsion in
abelian groups. Torsion theories for triangulated categories were
introduced in \cite{IyamaYoshino-mutation-in-tri-cats-rigid-CM-mods}
and used in the study of rigid Cohen-Macaulay modules over specific
Veronese subrings. Analogously, Nakaoka \cite{Nakaoka-cotorsion-pairs-I}
defined cotorsion pairs for triangulated categories as follows. Let
$\CC$ be a triangulated category with suspension functor $\sus$.
A \emph{cotorsion pair }on $\CC$ is a pair $(\CU,\CV)$ of full,
additive subcategories of $\CC$ that are closed under isomorphisms
and direct summands, satisfying $\Ext_{\CC}^{1}(\CU,\CV)=0$ and $\CC=\CU*\sus\CV$
(see Definitions \ref{def:CU*CV for subcats CU,CV of C} and \ref{def:cotorsion pair in tri'd cat}). This allowed
Nakaoka to extract an abelian category, known as the \emph{heart }of
the cotorsion pair \cite[Def. 3.7]{Nakaoka-cotorsion-pairs-I},
from the triangulated category. The key motivating examples for Nakaoka
were the following. \begin{enumerate}[(i)]
\item A $t$-structure $(\CC^{\leq 0},\CC^{\geq 0})$ on a triangulated category $\CC$, in the sense of \cite{BeilinsonBernsteinDeligne-perverse-sheaves}, can be interpreted as a cotorsion pair $(\sus\CC^{\leq 0},\sus^{-1}\CC^{\geq 0})$. In this case the heart  $\CC^{\leq0}\cap\CC^{\geq0}$ of the $t$-structure coincides with the heart of the cotorsion pair.

\item Suppose $\CC$ is a triangulated category, with a tilting subcategory $\CT$ (see \cite[Def. 2.2]{Iyama-higher-dimnl-AR-theory-maximal-orthog-subcats}). It was shown in \cite{KoenigZhu-from-tri-cats-to-abelian-cats} (see also \cite{KellerReiten-ct-algebras-are-gorenstein-stably-cy} and \cite{BuanMarshReiten-cluster-tilted-algebras}) that $\CC/[\CT]$ is an abelian category, where $[\CT]$ is the ideal of morphisms factoring through $\CT$. The corresponding cotorsion pair in this setting is $(\CT,\CT)$ and has $\CC/[\CT]$ as its heart.
\end{enumerate}

In \cite{BuanMarsh-BM2}, Buan and Marsh generalised the
results of \cite{KoenigZhu-from-tri-cats-to-abelian-cats} and
\cite{BuanMarshReiten-cluster-tilted-algebras} in the following
way. Assume $k$ is a field, and suppose $\CC$ is a skeletally small,
$\Hom$-finite, Krull-Schmidt, triangulated $k$-category has Serre duality (see Definition \ref{def:Serre functor}). For a subcategory
$\CA$ of $\CC$ that is closed under finite direct sums, let $[\CA]$ denote the ideal of $\CC$ consisting
of morphisms that factor through an object of $\CA$. Fix an object
$R$ of $\CC$ that is rigid (see \S\ref{sec:application to cluster category}).
It was shown \cite[Thm. 5.7]{BuanMarsh-BM2} that there
is an equivalence $(\CC/[\mathcal{X}_{R}])_{\CR}\simeq\mod(\End_{\CC}R)^{\text{op}}$,
where $\mathcal{X}_{R}=\Ker(\Hom_{\CC}(R,-))$ and $(\CC/[\mathcal{X}_{R}])_{\CR}$
is the (Gabriel-Zisman) localisation (see \cite{GabrielZisman-calc-of-fractions})
of $\CC/[\mathcal{X}_{R}]$ at the class $\CR$ of regular (see Remark
\ref{rem:semi-abelian iff parallels are regular}) morphisms. Beligiannis
further developed these ideas in \cite{Beligiannis-rigid-objects-triangulated-subfactors-and-abelian-localizations}.

Nakaoka was then able to put this into a more general context by introducing
the following concept in \cite{Nakaoka-twin-cotorsion-pairs}.
A \emph{twin cotorsion pair} on $\CC$ consists of two cotorsion pairs $(\CS,\CT)$
and $(\CU,\CV)$ on $\CC$ which satisfy $\CS\subseteq\CU$.
As for cotorsion pairs, Nakaoka defined the \emph{heart }of a twin
cotorsion pair as a certain subfactor category of $\CC$ (see Definition
\ref{def:quotients of C+ C- and H}). By setting $\CS=\CU$ and $\CT=\CV$,
one recovers the original cotorsion pair theory: the heart of the
twin cotorsion pair $((\CU,\CV),(\CU,\CV))$ coincides with the heart
of the cotorsion pair $(\CU,\CV)$ (see \cite[Exam. 2.10]{Nakaoka-twin-cotorsion-pairs}).

For a twin cotorsion pair, the associated heart $\overline{\CH}$
is shown \cite[Thm. 5.4]{Nakaoka-twin-cotorsion-pairs}
to be semi-abelian (see Definition \ref{def:left and right semi-abelian}).
Furthermore, Nakaoka showed \cite[Thm. 6.3]{Nakaoka-twin-cotorsion-pairs}
that if $\CU\subseteq\CS*\CT$ or $\CT\subseteq\CU*\CV$, then $\overline{\CH}$
is integral (see Definition \ref{def:integral category}), so that
localising at the class of regular morphisms produces an abelian category
(see \cite{Rump-almost-abelian-cats}). With $\CC$
and $R$ as above, and setting $((\CS,\CT),(\CU,\CV))=((\add\sus R,\CX_{R}),(\XR,\tensor[]{\XR}{^{\perp_1}}))$,
where $\tensor[]{\XR}{^{\perp_1}}=\Ker(\Ext^{1}_{\CC}(\CX_R,-))$,
one obtains the aforementioned result \cite[Thm. 5.7]{BuanMarsh-BM2}
of Buan and Marsh (see Lemma \ref{lem:Nakaoka twin cotorsion pairs Example 2.10 (2) twin cot pair giving cluster cat mod X_R}).

The main result of this article concerns quasi-abelian categories;
a \emph{quasi-abelian }category is an additive category which has
kernels and cokernels, and in which kernels are stable under pushout
and cokernels are stable under pullback (see Definition \ref{def:stable under PB (PO)}).
Important examples of such categories include: the category of topological
abelian groups; the category of $\Lambda$-lattices for $\Lambda$
an order over a noetherian integral domain; any abelian category;
and the torsion class and torsion-free class in any torsion theory
of an abelian category (see \cite[\S2]{Rump-almost-abelian-cats}
for more details). In this article, we prove that the heart of a twin
cotorsion pair, satisfying a different mild assumption, is quasi-abelian
(see Theorem \ref{thm:CC tri'd cat, CH =00003D CC^- or CC^+ then stable CH is quasi-abelian}).
This assumption is satisfied if $\CU\subseteq\CT$ or $\CT\subseteq\CU$,
and hence is met in the setting of \cite{BuanMarsh-BM2}
discussed above (see Corollary \ref{cor:CC tri'd, twin cotorsion pair, if CU in CT or CT in CU then stable heart is both integral and quasi-abelian})
where $\CT=\CU$.

Let $((\CS,\CT),(\CU,\CV))$ be a twin cotorsion pair with heart $\overline{\CH}$
on a Krull-Schmidt, triangulated category. We show in \S\ref{sec:Localisation of integral heart}
that if $\CT$ coincides with $\CU$, then the heart $\overline{\mathscr{H}}_{(\CS,\CT)}$
of $(\CS,\CT)$ (see \cite[Def. 3.7]{Nakaoka-cotorsion-pairs-I})
is equivalent to the localisation $\overline{\CH}_{\CR}$ of $\overline{\CH}$
at the class $\CR$ of its regular morphisms (see Theorem \ref{thm:CT=00003DCU implies G-Z localisation of CH at regulars is equivalent to heart of cotorsion pair (CS,CT)}).
Since $\CT=\CU$ when $((\CS,\CT),(\CU,\CV))=((\add\sus R,\CX_{R}),(\XR,\tensor[]{\XR}{^{\perp_1}}))$,
the results of \S\ref{sec:Localisation of integral heart} also apply
in the setting of Buan and Marsh as we explain in \S\ref{sec:application to cluster category}.
Our methods are also related to work of Marsh and Palu: in \cite{MarshPalu-nearly-morita-equivalences-and-rigid-objects},
equivalences are found from subfactor categories of a Krull-Schmidt,
$\Hom$-finite, triangulated category to localisations of module (and
hence abelian) categories, whereas we localise not necessarily abelian
categories. We also note that Theorem \ref{thm:CT=00003DCU implies G-Z localisation of CH at regulars is equivalent to heart of cotorsion pair (CS,CT)}
may be obtained from results of \cite{Beligiannis-rigid-objects-triangulated-subfactors-and-abelian-localizations}
in a different way (see Remark \ref{rem:comparison to Bel13}).

In particular, the cluster category $\CC_{H}$ (see \cite{BMRRT-cluster-combinatorics},
\cite{CalderoChapotonSchiffler-quivers-arising-from-cluster-a_n-case}) associated to a hereditary
$k$-algebra $H$ is an example of a $\Hom$-finite, Krull-Schmidt,
triangulated $k$-category that has Serre duality, and this is
the motivation for our results (see Example \ref{exa:C of A_4 example}).
It is especially interesting that $\CC/[\CX_{R}]$ is quasi-abelian
in this case, as many aspects of Auslander-Reiten theory for abelian
categories (developed in \cite{AuslanderReiten-Rep-theory-of-Artin-algebras-IV},
\cite{AuslanderReiten-Rep-theory-of-Artin-algebras-V}) still
apply for quasi-abelian categories (see the forthcoming preprint \cite{Shah-AR-theory-quasi-abelian-cats-KS-cats}).

This paper is organised in the following way. We first recall the
notion of a quasi-abelian category in \S\ref{sub:Preabelian-Categories},
then the definition and some properties of twin cotorsion pairs as
well as some new observations in \S\ref{sub:Twin-Cotorsion-Pairs}.
In \S\ref{sec:Main-Result} we prove our main result: the case when
the heart of a twin cotorsion pair becomes quasi-abelian. In \S\ref{sec:Localisation of integral heart}
we relate the heart of a twin cotorsion pair $((\CS,\CT),(\CU,\CV))$
to the heart of the cotorsion pair $(\CS,\CT)$ whenever $\CT=\CU$.
Lastly, we explore our main motivating example in \S\ref{sec:application to cluster category},
namely the setting of \cite{BuanMarsh-BM2}.

\section{\label{sec:Preliminaries}Preliminaries}

\subsection{\label{sub:Preabelian-Categories}Preabelian categories}

The main result of this paper concerns a type of category more general
than an abelian category---namely a \emph{quasi-abelian }category.
However, before giving the definition of such a category, we recall
some preliminary definitions. We only give a quick summary of the
theory and for more details we refer the reader to \cite{Rump-almost-abelian-cats}.
\begin{defn}
\label{def:preabelian cat}\cite[p. 24]{Popescu-abelian-cats-with-apps-to-rings-and-modules},
\cite[\S5.4]{BucurDeleanu-intro-to-theory-of-cats-and-functors}
A \emph{preabelian }category is an additive category in which every
morphism has a kernel and a cokernel.
\end{defn}

\begin{defn}
\label{def:image and coimage}\cite[p. 23]{Popescu-abelian-cats-with-apps-to-rings-and-modules}
Given a morphism $f\colon A\to B$ in a category $\CA$, the \emph{coimage
}$\coim f\colon A\to\Coim f$, if it exists, is the
cokernel $\cok(\ker f)$ of the kernel of $f$. Similarly, the \emph{image
}$\image f\colon\Image f\to B$ is the kernel $\ker(\cok f)$ of the
cokernel of $f$.
\end{defn}
The following proposition is then easily checked.
\begin{prop}
\label{prop:decomposition of map in preab cat}\emph{\cite[p. 24]{Popescu-abelian-cats-with-apps-to-rings-and-modules}}
Let $\CA$ be a preabelian category and $f\colon A\to B$ a morphism
in $\CA$. Then $f$ decomposes as$$\begin{tikzcd}[column sep=1.5cm] A \arrow[two heads]{d}[swap]{\coim f} \arrow{r}{f}\commutes{dr} & B \\ \Coim f  \arrow{r}[swap]{\widetilde{f}}& \Image f \arrow[hook]{u}[swap]{\image f}\end{tikzcd}$$\end{prop}
\begin{defn}
\label{def:parallel of morphism f and strict morphism}\cite[p. 24]{Popescu-abelian-cats-with-apps-to-rings-and-modules}
The morphism $\widetilde{f}$ in Proposition \ref{prop:decomposition of map in preab cat}
above is called the \emph{parallel} \emph{of} $f$. Furthermore, if
$\widetilde{f}$ is an isomorphism then $f$ is said to be \emph{strict}.
\end{defn}

\begin{defn}
\label{def:left and right semi-abelian}\cite[p. 167]{Rump-almost-abelian-cats}
Let $\CA$ be a preabelian category. We call $\CA$ \emph{left semi-abelian
}if each morphism $f\colon A\to B$ factorises as $f=ip$ for some
monomorphism $i$ and cokernel $p$. We call $\CA$ \emph{right semi-abelian
}if instead each morphism $f$ decomposes as $f=ip$ with $i$ a kernel
and $p$ some epimorphism. If $\CA$ is both left and right semi-abelian,
then it is simply called \emph{semi-abelian}.\end{defn}
\begin{rem}
\label{rem:semi-abelian iff parallels are regular}A category is semi-abelian
if and only if, for every morphism $f$, the\emph{ }parallel $\widetilde{f}$
of $f$ is \emph{regular}, i.e. simultaneously monic and epic (see
\cite[pp. 167--168]{Rump-almost-abelian-cats}).
\end{rem}

\begin{defn}
\label{def:stable under PB (PO)}Let $\CA$ be a category, and suppose$$\begin{tikzcd}A\arrow{r}{a}\arrow{d}[swap]{b} & B \arrow{d}{c}\\ C \arrow{r}[swap]{d} & D\end{tikzcd}$$is
a commutative diagram in $\CA$. Let $\CP$ be a class of morphisms in $\CA$ (e.g. the class of all kernels in $\CA$). We say that $\CP$ is
\emph{stable under pullback} (respectively, \emph{stable under pushout})
if, in any diagram above that is a pullback (respectively, pushout) square,
$d$ is in $\CP$ implies that $a$ is in $\CP$ (respectively, $a$ is in 
$\CP$ implies that $d$ is in $\CP$).
\end{defn}

\begin{defn}
\label{def:quasi-abelian category}\cite[p. 168]{Rump-almost-abelian-cats}
Let $\CA$ be a preabelian category. We call $\CA$ \emph{left quasi-abelian} 
if cokernels are stable under pullback in $\CA$. If kernels are stable
under pushout in $\CA$, then we call $\CA$ \emph{right quasi-abelian. }Furthermore,
if $\CA$ is left and right quasi-abelian, then $\CA$ is simply called
\emph{quasi-abelian}.\end{defn}
\begin{rem}
The history of the term `quasi-abelian' category is not straightforward.
We use the terminology as in \cite{Rump-counterexample-to-Raikov},
but note that such categories were called `almost abelian' in \cite{Rump-almost-abelian-cats}.
We refer the reader to the `Historical remark' in \cite{Rump-counterexample-to-Raikov}
for more details.
\end{rem}

\begin{rem}
It is also worth remarking that a category is abelian if and only
if it is a quasi-abelian category in which every morphism is strict.\end{rem}
\begin{defn}
\label{def:integral category}\cite[p. 168]{Rump-almost-abelian-cats}
Let $\CA$ be a preabelian category. We call $\CA$ \emph{left integral} if epimorphisms
are stable under pullback in $\CA$. If monomorphisms
are stable under pushout in $\CA$, then we call $\CA$ \emph{right integral}. If $\CA$ is both left and right integral, then $\CA$ is called \emph{integral}.
\end{defn}
Lastly in this section, we recall an observation from \cite{Rump-almost-abelian-cats}.
\begin{prop}
\emph{\cite[p. 169, Cor. 1]{Rump-almost-abelian-cats}}
Every left (respectively, right) quasi-abelian or left (respectively,
right) integral category is left (respectively, right) semi-abelian.
\end{prop}

\subsection{\label{sub:Twin-Cotorsion-Pairs}Twin cotorsion pairs on triangulated
categories}

Throughout this section, let $\CC$ denote a fixed triangulated category
with suspension functor $\sus$. We use the labelling of the axioms
of a triangulated category as in \cite{HolmJorgensen-tri-cats-intro},
and its distinguished triangles will just be called triangles. We
follow \cite{Nakaoka-cotorsion-pairs-I}
and \cite{Nakaoka-twin-cotorsion-pairs}
in order to recall some of the definitions and theory concerning twin
cotorsion pairs on triangulated categories, but first we need some
notation.
\begin{defn}
\label{def:i-th perp definitions (Hom or Ext^i perps)}Let $\CU\subseteq\CC$
be a full, additive subcategory of $\CC$ that is closed under isomorphisms
and direct summands. By $\Ext_{\CC}^{i}(\CU,X)=0$ (respectively,
$\Ext_{\CC}^{i}(X,\CU)=0$) we mean $\Ext_{\CC}^{i}(U,X)=0$ (respectively,
$\Ext_{\CC}^{i}(X,U)=0$) for all $U\in\CU$. We define the following
full, additive subcategories of $\CC$ where $i\in\BN$:
\[
\CU^{\perp_{i}}\deff\left\{ X\in\CC\mid\Ext_{\CC}^{i}(\CU,X)=0\right\} ,
\]
\[
^{\perp_{i}}\CU\deff\left\{ X\in\CC\mid\Ext_{\CC}^{i}(X,\CU)=0\right\} .
\]

\end{defn}

\begin{defn}
\label{def:CU*CV for subcats CU,CV of C}\cite[p. 122]{IyamaYoshino-mutation-in-tri-cats-rigid-CM-mods}
Let $\CU,\CV\subseteq\CC$ be full, additive subcategories of $\CC$
that are closed under isomorphisms and direct summands. By $\CU*\CV$
we denote the full subcategory of $\CC$ consisting of objects $X\in\CC$
for which there exists a triangle $U\to X\to V\to\sus U$ in $\CC$
with $U\in\CU$, $V\in\CV$.
\end{defn}

\begin{defn}
\label{def:cotorsion pair in tri'd cat}\cite[Def. 2.1]{Nakaoka-cotorsion-pairs-I}
Let $\CU,\CV\subseteq\CC$ be full, additive subcategories of $\CC$
that are closed under isomorphisms and direct summands. We call (the
ordered pair) $(\CU,\CV)$ a \emph{cotorsion pair (on $\CC$) }if
$\Ext_{\CC}^{1}(\CU,\CV)=0$ and $\CC=\CU*\sus\CV$.
\end{defn}
As pointed out in \cite[Rem. 2.2]{Nakaoka-cotorsion-pairs-I},
a pair $(\CU,\CV)$ is a cotorsion pair on a Krull-Schmidt, $\Hom$-finite, triangulated $k$-category $\CC'$ (with suspension $\sus '$) if and only if $(\sus'^{-1}\CU,\CV)$
is a torsion theory in $\CC'$ as defined in \cite{IyamaYoshino-mutation-in-tri-cats-rigid-CM-mods}. Recall that a \emph{torsion theory in $\CC'$} (in the sense of \cite[Def. 2.2]{IyamaYoshino-mutation-in-tri-cats-rigid-CM-mods}) is a pair $(\CX,\CY)$ of full additive subcategories $\CX,\CY$ of $\CC'$ that are closed under isomorphisms and direct summands, such that $\Hom_{\CC'}(\CX,\CY)=0$ and $\CC'=\CX*\CY$.
We note that in \cite{IyamaYoshino-mutation-in-tri-cats-rigid-CM-mods} all categories are assumed to be Krull-Schmidt and all triangulated categories are also assumed to be $\Hom$-finite $k$-categories (see \cite[pp. 121--122]{IyamaYoshino-mutation-in-tri-cats-rigid-CM-mods}). Therefore, some of the results from \cite{IyamaYoshino-mutation-in-tri-cats-rigid-CM-mods} may not translate directly over to the more general setting considered in \cite{Nakaoka-twin-cotorsion-pairs}.
\begin{defn}
\label{def:right left minimal morphism}\cite[\S2]{AuslanderReiten-Rep-theory-of-Artin-algebras-IV}
Let $f\colon X\to Y$ be a morphism in $\CC$. We say that $f$ is
\emph{right minimal }(respectively, \emph{left minimal}) if, for any
endomorphism $g\colon X\to X$ (respectively, $g\colon Y\to Y$),
$fg=f$ (respectively, $gf=f$) implies $g$ is an automorphism.
\end{defn}

\begin{defn}
\cite[p. 114]{AuslanderReiten-apps-of-contravariantly-finite-subcats}
Let $\CX\subseteq\CC$ be a full subcategory, closed under isomorphisms
and direct summands. \begin{enumerate}[(i)]
\item A \emph{right} $\CX$\emph{-approximation} of $A$ in $\CC$ is a morphism $X\to A$ in $\CC$ with $X\in\CX$, such that for any object $X'\in \CX$ we have an exact sequence $$\Hom_{\CA}(X',X)\to\Hom_{\CA}(X',A)\to 0.$$A right $\CX$-approximation is called a \emph{minimal} right $\CX$-approximation if it is also right minimal. 
\item A \emph{left} $\CX$\emph{-approximation} of $A$ in $\CC$ is a morphism $A\to X$ in $\CC$ with $X\in\CX$, such that for any object $X'\in \CX$ we have an exact sequence $$\Hom_{\CA}(X,X')\to\Hom_{\CA}(A,X')\to 0.$$A left $\CX$-approximation is called a \emph{minimal} left $\CX$-approximation if it is also left minimal.
\end{enumerate}
\end{defn}
The terminology of approximations was introduced in \cite{AuslanderReiten-apps-of-contravariantly-finite-subcats},
but the same notions were established independently by Enochs \cite{Enochs-inj-flat-covers-envelopes-and-resolvents}
specifically for the subcategories of injective objects and projective
objects in a module category. The term `preenvelope' (respectively,
`precover') in \cite{Enochs-inj-flat-covers-envelopes-and-resolvents}
corresponds to the notion of left\emph{ }(respectively, right) approximation.
\begin{lem}[Triangulated Wakamatsu's Lemma]
\label{lem:triangulated wakamatsu}\emph{\cite[Lem. 2.1]{Jorgensen-auslander-reiten-triangles-in-subcategories}}
Let $\CX$ be an extension-closed full subcategory of $\CC$ that is closed under isomorphisms and direct summands.\emph{\begin{enumerate}[(i)]
\item \emph{Suppose} $X\overset{x}{\longrightarrow}A$ \emph{is a minimal right} $\CX$\emph{-approximation of} $A$ \emph{in} $\CC$\emph{, which completes to a triangle} $\sus^{-1}A \overset{w}{\longrightarrow}Y \longrightarrow X \overset{x}{\longrightarrow}A$\emph{. Then} $w\colon\sus^{-1}A\to Y$ \emph{is a left} $\CX^{\perp_{1}}$\emph{-approximation of} $\sus^{-1}A$\emph{.}
\item \emph{Suppose} $A\overset{x'}{\longrightarrow}X'$ \emph{is a minimal left} $\CX$\emph{-approximation of} $A$ \emph{in} $\CC$\emph{, which completes to a triangle} $A \overset{x'}{\longrightarrow}X' \longrightarrow Z \overset{z}{\longrightarrow}\sus A$\emph{. Then} $z\colon Z\to \sus A$ \emph{is a right} $\tensor[^{\perp_{1}}]{\CX}{}$\emph{-approximation of} $\sus A$\emph{.}
\end{enumerate}}
\end{lem}
Although the notion of a contravariantly (respectively, covariantly)
finite subcategory (see below) is related to the idea of right (respectively,
left) approximations\emph{, }it dates back to \cite[p. 81]{AuslanderSmalo-preprojective-modules}
in which these concepts were defined in the context of module categories.
\begin{defn}
\label{def:contra-co-functorially finite}\cite[pp. 114, 142]{AuslanderReiten-apps-of-contravariantly-finite-subcats}
Let $\CX\subseteq\CC$ be a full subcategory, closed under isomorphisms
and direct summands. We say $\CX$ is \emph{contravariantly} (respectively,
\emph{covariantly}) \emph{finite} if $A$ has a right (respectively,
left) $\CX$-approximation for each $A\in\CC$. If $\CX$ is both
contravariantly finite and covariantly finite, then $\CX$ is called
\emph{functorially finite}.
\end{defn}
The next proposition collects some elementary properties about
cotorsion pairs that will be very useful in the sequel; see for example
\cite{IyamaYoshino-mutation-in-tri-cats-rigid-CM-mods}
or \cite{Nakaoka-cotorsion-pairs-I}.
Recall that if $\CU$ is a full, additive subcategory of $\CC$, then
$\add\CU$ denotes the full, additive subcategory of $\CC$ that consists
of objects of $\CC$ which are isomorphic to direct summands of finite
direct sums of objects of $\CU$.
\begin{prop}
\label{prop:elementary properties for cotorsion pair}Let $(\CU,\CV)$
be a cotorsion pair\emph{ on $\CC$.\begin{enumerate}[(i)]
\item \cite[p. 123]{IyamaYoshino-mutation-in-tri-cats-rigid-CM-mods}\emph{,} \cite[Rem. 2.3]{Nakaoka-cotorsion-pairs-I} \emph{We have} $\CU = \tensor[^{\perp_1}]{\CV}{}$ \emph{and} $\CV=\CU^{\perp_1}$\emph{.}
\item \cite[p. 123]{IyamaYoshino-mutation-in-tri-cats-rigid-CM-mods}\emph{,} \cite[Lem. 2.14]{Nakaoka-twin-cotorsion-pairs} \emph{Let} $X$ be an object in $\CC$\emph{. Since} $(\CU,\CV)$ \emph{is a cotorsion pair, there is a triangle} $U\overset{u}{\longrightarrow}X\overset{v}{\longrightarrow}\sus V\overset{w}{\longrightarrow}\sus U$\emph{, where} $U\in\CU$ \emph{and} $V\in\CV$\emph{. Then the morphism} $u\colon U\to X$ \emph{is a right} $\CU$\emph{-approximation of} $X$ \emph{and the morphism} $v\colon X\to\sus V$ \emph{is a left} $\sus\CV$\emph{-approximation of} $X$\emph{.}
\item \emph{The subcategory} $\CU$ \emph{is contravariantly finite and the subcategory} $\CV$ \emph{is covariantly finite.}
\item \cite[Rem. 2.4]{Nakaoka-cotorsion-pairs-I} \emph{The subcategories} $\CU$ \emph{and} $\CV$ \emph{are extension-closed.}
\end{enumerate}}
\end{prop}

\begin{defn}
\label{def:twin cotorsion pair in tri'd cat}\cite[Def. 2.7]{Nakaoka-twin-cotorsion-pairs}
Let $(\CS,\CT)$ and $(\CU,\CV)$ be two cotorsion pairs on $\CC$.
The ordered pair $((\CS,\CT),(\CU,\CV))$ is called a \emph{twin cotorsion
pair (on $\CC$) }if $\Ext_{\CC}^{1}(\CS,\CV)=0$.
\end{defn}
The following easily verifiable result is often useful.
\begin{prop}
\emph{\cite[p. 198]{Nakaoka-twin-cotorsion-pairs}}
Let $(\CS,\CT)$ and $(\CU,\CV)$ be cotorsion pairs on $\CC$. Then
$((\CS,\CT),(\CU,\CV))$ is a twin cotorsion pair $\iff$ $\CS\subseteq\CU$
$\iff$ $\CV\subseteq\CT$.\end{prop}
Throughout the remainder of this section, let $((\CS,\CT),(\CU,\CV))$ be a twin cotorsion pair on $\CC$.
\begin{defn}
\cite[Def. 2.8]{Nakaoka-twin-cotorsion-pairs}
We define full
subcategories of $\CC$ as follows:
\[
\begin{array}{cccc}
\CW\deff\CT\cap\CU, & \CC^{-}\deff\sus^{-1}\CS*\CW, & \CC^{+}\deff\CW*\sus\CV, & \CH\deff\CC^{-}\cap\CC^{+}.\end{array}
\]

\end{defn}
From this definition, we immediately see that $\CW$ is contained
in the subcategories $\CC^{-},$ $\CC^{+}$ and $\CH$; and that $\CW$
is extension-closed as $\CT$ and $\CU$ are extension-closed. It is also clear
that $\CW$, $\CC^{-}$, $\CC^{+}$ and $\CH$ are additive and closed under isomorphisms.
\begin{prop}
The subcategories
$\CW$, $\CC^{-}$, $\CC^{+}$ and $\CH$ are all closed under direct
summands.\end{prop}
\begin{proof}
Since $\CT$ and $\CU$ are assumed to be closed under direct summands
(see Definition \ref{def:cotorsion pair in tri'd cat}), we immediately
see that $\CW$ is also closed under direct summands. That $\CH$ is closed
under direct summands will follow from $\CC^{-}$ and $\CC^{+}$ having this
property. We will give the proof just for $\CC^{-}$ as the proof
for $\CC^{+}$ is similar.

Suppose $X=X_{1}\oplus X_{2}\in\CC^{-}$, then there is a distinguished
triangle $\sus^{-1}S \overset{s}{\longrightarrow}X\overset{x}{\longrightarrow}W\overset{t}{\longrightarrow}S$ with $S\in\CS$ and $W\in\CW$. Since $\CC=\CS*\sus\CT$, there exists
a triangle $\begin{tikzcd}[column sep=1.3em] \sus^{-1}S_1\arrow{r}{a}&X_1\arrow{r}{b}&T_1\arrow{r}{c}& S_1 \end{tikzcd}$
where $S_{1}\in\CS$ and $T_{1}\in\CT$. Thus, it suffices to show
that $T_{1}\in\CU$ as then we will have $T_{1}\in\CT\cap\CU=\CW$, and hence $T_{1}\in\sus^{-1}\CS*\CW=\CC^{-}$.

First, we claim that $x\colon X\to W$ is a left $\CT$-approximation
of $X$. Indeed, if $T\in\CT$ then we get an exact sequence $\Hom_{\CC}(W,T)\to\Hom_{\CC}(X,T)\to\Hom_{\CC}(\sus^{-1}S,T)$
since $\Hom_{\CC}(-,T)$ is a cohomological functor (see \cite[Prop. I.1.2]{Happel-triangulated-cats-in-rep-theory}),
where $\Hom_{\CC}(\sus^{-1}S,T)\iso\Hom_{\CC}(S,\sus T)=\Ext_{\CC}^{1}(S,T)=0$
since $(\CS,\CT)$ is a cotorsion pair. Thus, $\Hom_{\CC}(W,T)\to\Hom_{\CC}(X,T)$
is surjective for $T\in\CT$ and $x\colon X\to W$ is a left $\CT$-approximation
of $X$. 

In order to show $T_{1}\in\CU,$ it is enough to show that any $v\colon T_{1}\to\sus V$
is in fact the zero map as $\CU=\tensor[^{\perp_{1}}]{\CV}{}$ (see Proposition \ref{prop:elementary properties for cotorsion pair}). Let
$v\colon T_{1}\to\sus V$ be arbitrary. Since $b\pi_{1}\colon X=X_{1}\oplus X_{2}\to T_{1}$
is a morphism with codomain in $\CT$, where $\pi_{1}\colon X\to X_{1}$
is the canonical projection, it must factor through the left $\CT$-approximation
$x\colon X\to W$. That is, there exists $d\colon W\to T_{1}$ such
that $dx=b\pi_{1}$. We then have $(vb)\pi_{1}=vdx=0$, because $vd\colon W\to\sus V$
vanishes as $W\in\CW\subseteq\CU=\tensor[^{\perp_{1}}]{\CV}{}$. This in turn implies $vb=0$ as $\pi_{1}$ is an epimorphism.
Since $\begin{tikzcd}[column sep=1.3em] \sus^{-1}S_{1}\arrow{r}{a}&X_{1}\arrow{r}{b}&T_{1}\arrow{r}{c}& S_{1} \end{tikzcd}$
is a triangle, we see that $v\colon T_{1}\to\sus V$ must factor through
$c\colon T_{1}\to S_{1}$. Thus, $v=fc$ for some $f\in\Hom_{\CC}(S_{1},\sus V)=\Ext_{\CC}^{1}(S_{1},V)=0$
by definition of a twin cotorsion pair. Hence, $v=0$ and we are done.
\end{proof}
We now recall some notions from \cite{Nakaoka-twin-cotorsion-pairs}
needed for the remainder of this section.
\begin{defn}
\label{def:k_X for twin cotorsion pair}\cite[Def. 3.1]{Nakaoka-twin-cotorsion-pairs}
For
$X\in\CC$, we define $K_{X}\in\CC$ and a morphism $k_{X}\colon K_{X}\to X$
as follows. Since $\CS*\sus\CT=\CC=\CU*\sus\CV$, we have two triangles
$\begin{tikzcd}[column sep=1.3em] \sus^{-1}S\arrow{r}{}&X\arrow{r}{a}&T\arrow{r}{}& S \end{tikzcd}$
$(S\in\CS$, $T\in\CT)$ and $\begin{tikzcd}[column sep=1.3em] U\arrow{r}{}&T\arrow{r}{b}&\sus V\arrow{r}{}& \sus U \end{tikzcd}$
$(U\in\CU$, $V\in\CV)$. Then we may complete the composition $ba\colon X\to\sus V$
to a triangle $$\begin{tikzcd}V \arrow{r}& K_X \arrow{r}{k_X}& X \arrow{r}{ba}& \sus V.\end{tikzcd}$$
\end{defn}

\begin{defn}
\label{def:z_X for twin cotorsion pair}\cite[Def. 3.4]{Nakaoka-twin-cotorsion-pairs}
For
$X\in\CC$, we define $Z_{X}\in\CC$ and $z_{X}\colon X\to Z_{X}$
as follows. Since $\CS*\sus\CT=\CC=\CU*\sus\CV$, we have two triangles
$\begin{tikzcd}[column sep=1.3em]V \arrow{r}{}&U\arrow{r}{c}&X\arrow{r}{}&\sus V \end{tikzcd}$
$(U\in\CU$, $V\in\CV)$ and $\begin{tikzcd}[column sep=1.3em] \sus^{-1}T\arrow{r}{}&\sus^{-1}S\arrow{r}{d}&U\arrow{r}{}&T \end{tikzcd}$
$(S\in\CS$, $T\in\CT)$. Then we may complete the composition $cd\colon\sus^{-1}S\to X$
to a triangle $$\begin{tikzcd}\sus^{-1}S\arrow{r}{cd}& X \arrow{r}{z_X}& Z_X \arrow{r}& S.\end{tikzcd}$$
\end{defn}

\begin{defn}
\label{def:m_f for twin cotorsion pair}\cite[Def. 4.1]{Nakaoka-twin-cotorsion-pairs}
Let
$f\colon X\to Y$ be a morphism in $\CC$ with $X\in\CC^{-}$. We
define $M_{f}\in\CC$ and $m_{f}\colon Y\to M_{f}$ as follows. Since
$X\in\CC^{-}$, there is a triangle $\begin{tikzcd}[column sep=1.3em] \sus^{-1}S\arrow{r}{s}&X\arrow{r}{}&W\arrow{r}{}&S \end{tikzcd}$
$(S\in\CS$, $W\in\CW)$. Then we may complete $fs\colon\sus^{-1}S\to Y$
to a triangle $$\begin{tikzcd}\sus^{-1}S \arrow{r}{fs}& Y \arrow{r}{m_f}& M_f \arrow{r} & S.\end{tikzcd}$$
\end{defn}

\begin{defn}
\label{def:l_f for twin cotorsion pair}\cite[Rem. 4.3]{Nakaoka-twin-cotorsion-pairs}
Let
$f\colon X\to Y$ be a morphism in $\CC$ with $Y\in\CC^{+}$. We
define $L_{f}\in\CC$ and $l_{f}\colon L_{f}\to X$ as follows. Since
$Y\in\CC^{+}$, there is a triangle $\begin{tikzcd}[column sep=1.3em] W\arrow{r}{}&Y\arrow{r}{v}&\sus V\arrow{r}{}&\sus W \end{tikzcd}$
$(W\in\CW$, $V\in\CV)$. Then we may complete $vf\colon X\to\sus V$
to a triangle $$\begin{tikzcd} V\arrow{r}&L_f\arrow{r}{l_f}&X\arrow{r}{vf}&\sus V.\end{tikzcd}$$
\end{defn}
We now present strengthened versions of \cite[Claim 3.2]{Nakaoka-twin-cotorsion-pairs}
and \cite[Claim 3.5]{Nakaoka-twin-cotorsion-pairs}.
\begin{prop}
Suppose $C$ is an arbitrary object of $\CC$. Then\emph{\begin{enumerate}[(i)]
\item $K_C\in\CC^{-}$\emph{;}
\item $C\in\CC^{+} \iff K_C\in\CC^{+} \iff K_C\in\CH$\emph{;}
\item $Z_C\in\CC^{+}$\emph{; and}
\item $C\in\CC^{-} \iff Z_C\in\CC^{-} \iff Z_C\in\CH$\emph{.}
\end{enumerate}}\end{prop}
\begin{proof}
The proofs for (i) and (iii) are \cite[Claim 3.2 (1)]{Nakaoka-twin-cotorsion-pairs}
and \cite[Claim 3.5 (1)]{Nakaoka-twin-cotorsion-pairs},
respectively. Since $K_{C}\in\CC^{-}$, we immediately see that $K_{C}\in\CC^{+}$
if and only if $K_{C}\in\CH$. For (ii), the proof that $C\in\CC^{+}$
implies $K_{C}\in\CH$ is \cite[Claim 3.2 (2)]{Nakaoka-twin-cotorsion-pairs}.
Thus, we show the converse. There is a triangle $V\to K_{C}\to C\to\sus V$
where $V\in\CV$, so if $K_{C}\in\CC^{+}$ then $C\in\CC^{+}$ using
\cite[Lem. 2.13 (2)]{Nakaoka-twin-cotorsion-pairs}.
The proof of statement (iv) is similar.
\end{proof}
The next proposition follows from \cite[Prop. 3.6]{Nakaoka-twin-cotorsion-pairs}
and \cite[Prop. 3.7]{Nakaoka-twin-cotorsion-pairs},
but we state it in the language of approximations.
\begin{prop}
Suppose $C$ is an arbitrary object of $\CC$. \emph{\begin{enumerate}[(i)]
\item \emph{The morphism} $k_C\colon K_C \to C$ \emph{is a right} $\CC^{-}$\emph{-approximation of} $C$\emph{.}
\item \emph{The morphism} $z_C\colon C \to Z_C$ \emph{is a left} $\CC^{+}$\emph{-approximation of} $C$\emph{.}
\end{enumerate}}
\end{prop}
For a subcategory $\CA\subseteq\CC$ that is closed under finite direct sums, we will denote by $[\CA]$ the
two-sided ideal of $\CC$ such that $[\CA](X,Y)$ consists of all
morphisms $X\to Y$ that factor through an object in 
$\CA$. Note that if $\CA$ is a full, additive subcategory that is
closed under isomorphisms and direct summands, then $[\CA]$ coincides
with the ideal generated by identity morphisms $1_{A}$ such that
$A\in\CA$. With this notation we are in position to recall the definition
of the heart associated to a twin cotorsion pair. 
\begin{defn}
\label{def:quotients of C+ C- and H}\cite{Nakaoka-twin-cotorsion-pairs}
Recall
that $\CW$ is a subcategory of $\CC^{+}$, $\CC^{-}$ and $\CH$.
We define the following additive quotients $\overline{\CC^{+}}\deff\CC^{+}/[\CW]$,
$\overline{\CC^{-}}\deff\CC^{-}/[\CW]$ and $\overline{\CH}\deff\CH/[\CW]$.
We call the category $\overline{\CH}$ the \emph{heart }of $((\CS,\CT),(\CU,\CV))$
by analogy with \cite{Nakaoka-cotorsion-pairs-I}.
\end{defn}
Suppose $\CI$ an ideal of $\CC$. We will denote by $\overline{f}$
the coset $f+\CI(X,Y)$ in $\Hom_{\CC/\CI}(X,Y)$ of $f\in\Hom_{\CC}(X,Y)$.
The next result is a combination of \cite[Cor. 4.5]{Nakaoka-twin-cotorsion-pairs}
and \cite[Cor. 4.6]{Nakaoka-twin-cotorsion-pairs},
and most of the proof can be found there. We provide the missing link.
\begin{prop}
\label{prop:twin cotorsion pair, f in CH(A,B) is epic in quotient iff Z_M_f in CW iff M_f in CU iff next map in triangle for f factors through CU}Let $f\in\Hom_{\CH}(A,B)$ be a morphism in the subcategory $\CH$, which completes to a triangle $\begin{tikzcd}[column sep=1.3em] A\arrow{r}{f}&B\arrow{r}{g}&C\arrow{r}{}& \sus A\end{tikzcd}$ in $\CC$.
Then the following are equivalent:\emph{\begin{enumerate}[(i)]
\item $\overline{f}\in \Hom_{\overline{\CH}}(A,B)$ \emph{is an epimorphism;}
\item $Z_{M_{f}}\in\CW$\emph{, i.e.} $Z_{M_{f}}\iso 0$ \emph{in} $\overline{\CH}$\emph{;}
\item $M_f\in\CU$\emph{; and}
\item $g\colon B\to C$ \emph{factors through} $\CU$\emph{.}
\end{enumerate}}\end{prop}
\begin{proof}
The equivalence of (i) -- (iii) is \cite[Cor. 4.5]{Nakaoka-twin-cotorsion-pairs},
and (iv) implies (i) is \cite[Cor. 4.6]{Nakaoka-twin-cotorsion-pairs}.
We prove (iii) implies (iv). To this end, suppose $M_{f}\in\CU$.
From Definition \ref{def:m_f for twin cotorsion pair}, there are triangles $\begin{tikzcd}[column sep=1.4em] \sus^{-1}S_A\arrow{r}{s_A}&A\arrow{r}{w_A}&W_A\arrow{r}{}&S_A \end{tikzcd}$ and $\begin{tikzcd}[column sep=1.7em] \sus^{-1}S_A\arrow{r}{fs_A}&B\arrow{r}{m_f}&M_f\arrow{r}{}&S_A,\end{tikzcd}$ where $S_{A}\in\CS$, $W_{A}\in\CW$. Then $g(fs_{A})=0$ as $gf=0$,
so $g$ factors through $m_{f}\colon B\to M_{f}$ where $M_{f}\in\CU$
by assumption. Hence, $g$ admits a factorisation through $\CU$ as
desired.
\end{proof}
To be explicit, we state the dual in full.
\begin{prop}
\label{prop:twin cotorsion pair, f in CH(A,B) is monic in quotient iff K_L_f in CW iff L_f in CT iff previous map in triangle for f factors through CT}Let $f\in\Hom_{\CH}(A,B)$ be a morphism in the subcategory $\CH$, which completes to a triangle $\begin{tikzcd}[column sep=1.3em] \sus^{-1}C\arrow{r}{h}&A\arrow{r}{f}&B\arrow{r}{}&C\end{tikzcd}$ in $\CC$.
Then the following are equivalent:\emph{\begin{enumerate}[(i)]
\item $\overline{f}\in \Hom_{\overline{\CH}}(A,B)$ \emph{is a monomorphism;}
\item $K_{L_{f}}\in\CW$\emph{, i.e.} $K_{L_{f}}\iso 0$ \emph{in} $\overline{\CH}$\emph{;}
\item $L_f\in\CT$\emph{; and}
\item $h\colon \sus^{-1}C\to A$ \emph{factors through} $\CT$\emph{.}
\end{enumerate}}
\end{prop}
The last result of this section is an application of these previous
two propositions to the case of a \emph{degenerate} twin cotorsion
pair; that is, a twin cotorsion pair $((\CS,\CT),(\CU,\CV))$ where
$(\CS,\CT)=(\CU,\CV)$. One may recover results from \cite{Nakaoka-cotorsion-pairs-I}
about cotorsion pairs on triangulated categories through the theory
of twin cotorsion pairs developed in \cite{Nakaoka-twin-cotorsion-pairs}
using such degenerate twin cotorsion pairs (see \cite[Exam. 2.10 (1)]{Nakaoka-twin-cotorsion-pairs}).
We recall some definitions, which we will also need later, from \cite{Nakaoka-cotorsion-pairs-I}
for convenience.
\begin{defn}
\label{def:heart of individual cotorsion pair}\cite{Nakaoka-cotorsion-pairs-I}
Given a cotorsion pair $(\CU,\CV)$ on a triangulated category, define
$\mathscr{W}\deff\CU\cap\CV$, $\mathscr{C}^{-}\deff\sus^{-1}\CU*\mathscr{W}$
and $\mathscr{C}^{+}\deff\mathscr{W}*\sus\CV$. The \emph{heart }of
the (individual) cotorsion pair $(\CU,\CV)$ is defined to be\emph{
}$\overline{\mathscr{H}}_{(\CU,\CV)}\deff(\mathscr{C}^{-}\cap\mbox{\ensuremath{\mathscr{C}}}^{+})/[\mathscr{W}]$.
\end{defn}
It is easy to see that $\mathscr{W}=\CW$, $\mathscr{C}^{-}=\CC^{-}$
and $\mathscr{C}^{+}=\CC^{+}$ for a degenerate twin cotorsion pair
$((\CU,\CV),(\CU,\CV))$, and that $\overline{\mathscr{H}}_{(\CU,\CV)}$
coincides with the heart $\overline{\CH}$ of the twin cotorsion pair
$((\CU,\CV),(\CU,\CV))$.
\begin{cor}
Suppose we have a degenerate twin cotorsion pair $((\CU,\CV),(\CU,\CV))$
on $\CC$ and objects $X,Y\in\CH$. Assume 
$\begin{tikzcd}[column sep=1.3em] Z\arrow{r}{}&X\arrow{r}{f}&Y\arrow{r}{}&\sus Z \end{tikzcd}$ is a triangle in $\CC$.
\emph{\begin{enumerate}[(i)]
\item \emph{If} $\overline{f}$ \emph{is epic in} $\overline{\CH}$\emph{, then} $Z\in\CC^{-}$\emph{.}
\item \emph{If} $\overline{f}$ \emph{is monic in} $\overline{\CH}$\emph{, then} $\sus Z\in\CC^{+}$\emph{.}
\end{enumerate}}\end{cor}
\begin{proof}
We only prove (i) as (ii) is similar. Since $\overline{f}$ is epic
in $\overline{\CH}$, by Proposition \ref{prop:twin cotorsion pair, f in CH(A,B) is epic in quotient iff Z_M_f in CW iff M_f in CU iff next map in triangle for f factors through CU}
we have that $M_{f}\in\CU=\CS$. Recall from Definition \ref{def:m_f for twin cotorsion pair} that $M_{f}$ is obtained by taking a triangle $\sus^{-1}S\overset{s}{\to}X\to W\to S$ ($S\in\CS$, $W\in\CW$), which exists as $X\in\CH\subseteq\CC^{-}$, and then completing the composition $fs$ to a triangle $\begin{tikzcd}[column sep=0.7cm]\sus^{-1}S \arrow{r}{fs}& Y\arrow{r}{m_{f}} & M_{f}\arrow{r}{} & S.\end{tikzcd}$ Then applying the octahedral axiom (TR5), we get a 
commutative diagram 
$$\begin{tikzcd}
\sus^{-1}S \arrow{r}{s}\arrow[equals]{d}& X \arrow{r}\arrow{d}{f}& W\arrow{r}\arrow{d} & S \arrow[equals]{d}\\
\sus^{-1}S \arrow{r}{fs}\arrow{d}& Y \arrow{r}{m_{f}}\arrow[equals]{d}& M_{f} \arrow{r}\arrow{d}& S \arrow{d}\\
X \arrow{r}{f}\arrow{d}& Y \arrow{r}\arrow{d}& \sus Z \arrow{r}\arrow[equals]{d}& \sus X\arrow{d} \\
W \arrow{r}& M_{f} \arrow{r}& \sus Z \arrow{r}& \sus W
\end{tikzcd}$$
where the rows are triangles. Therefore, there is a
triangle $\sus^{-1}M_{f}\to Z\to W\to M_{f}$ by (TR3), where $\sus^{-1}M_{f}\in\sus^{-1}\CU$
and $W\in\CW$, so $Z\in\sus^{-1}\CU*\CW=\sus^{-1}\CS*\CW=\CC^{-}$
as $\CU=\CS$ for a degenerate twin cotorsion pair.
\end{proof}

\section{Main result: the case when \texorpdfstring{$\overline{\CH}$}{Hbar} is quasi-abelian}\label{sec:Main-Result}

Let $\CC$ be a fixed triangulated category with suspension functor
$\sus$, and suppose $((\CS,\CT),(\CU,\CV))$ is a twin cotorsion
pair on $\CC$. No other assumptions are made on $\CC$ in this section.
We recall two key results from \cite{Nakaoka-twin-cotorsion-pairs}
concerning the factor category $\overline{\CH}=\CH/[\CW]$. First,
it is shown that $\overline{\CH}$ is semi-abelian \cite[Thm. 5.4]{Nakaoka-twin-cotorsion-pairs},
and that, if $\CU\subseteq\CS*\CT$ or $\CT\subseteq\CU*\CV$, then
$\overline{\CH}$ is also integral \cite[Thm. 6.3]{Nakaoka-twin-cotorsion-pairs}.

In this section we prove our main result: if $\CH$ is equal to $\CC^{-}$
or $\CC^{+}$, then $\overline{\CH}=\CH/[\CW]$ is quasi-abelian.
In order to prove this, we need the following lemma.
\begin{lem}
\label{lem:CA left semi-abelian category pullback square A B C D, cokernel map X to B implies A to B is a cokernel}Let
$\CA$ be a left semi-abelian category. Suppose $$\begin{tikzcd} A \arrow{r}{a}\arrow{d}[swap]{b}\PB{dr}& B\arrow{d}{c} \\ C\arrow{r}[swap]{d} & D \end{tikzcd}$$is
a pullback diagram in $\CA$. Suppose we also have morphisms $x_{B}\colon X\to B$
and $x_{C}\colon X\to C$ such that $x_{B}$ is a cokernel and $$\begin{tikzcd} X \arrow{r}{x_B}\arrow{d}[swap]{x_C}\commutes{dr}& B\arrow{d}{c} \\ C\arrow{r}[swap]{d} & D \end{tikzcd}$$commutes.
Then $a\colon A\to B$ is also a cokernel in $\CA$.\end{lem}
\begin{proof}
From the assumptions, we obtain the following commutative diagram
$$\begin{tikzcd} X\arrow[bend left]{drr}{x_B}\arrow[bend right]{ddr}[swap]{x_C}\arrow[dotted]{dr}{\exists !e}&&\\
&A \arrow{r}{a}\arrow{d}[swap]{b}\PB{dr}& B\arrow{d}{c}\\
&C\arrow{r}[swap]{d} & D \end{tikzcd}$$using the universal property for the pullback because $cx_{B}=dx_{C}.$
Thus, $ae=x_{B}$ is a cokernel, and hence $a$ is cokernel in the
left semi-abelian category $\CA$ by \cite[Prop. 2]{Rump-almost-abelian-cats}.
\end{proof}
Dually, we also have the following.
\begin{lem}
\label{lem:CA right semi-abelian category pushout square A B C D, kernel map C to X implies C to D is a kernel}Let
$\CA$ be a right semi-abelian category. Suppose $$\begin{tikzcd} A \arrow{r}{a}\arrow{d}[swap]{b}& B\arrow{d}{c} \\ C\arrow{r}[swap]{d} & D \PO{ul}\end{tikzcd}$$is
a pushout diagram in $\CA$. Suppose we also have morphisms $x_{B}\colon B\to X$
and $x_{C}\colon C\to X$ such that $x_{C}$ is a kernel and $$\begin{tikzcd} 
A \arrow{r}{a}\arrow{d}[swap]{b}\commutes{dr}& B\arrow{d}{x_B} \\
C \arrow{r}[swap]{x_C}& X 
\end{tikzcd}$$commutes. Then $d\colon C\to D$ is also a kernel in $\CA$.
\end{lem}
We will also need the following easy lemma in the proof of our main
theorem.
\begin{lem}
\label{lem:(epi) mono weak (co)kernel is (co)kernel}\emph{\cite[Lem. 2.5]{BuanMarsh-BM2}}
In any preadditive category $\CA$, we have the following. \emph{\begin{enumerate}[(i)]
\item \emph{A monomorphism that is a weak kernel is a kernel.}
\item \emph{An epimorphism that is a weak cokernel is a cokernel.}
\end{enumerate}}
\end{lem}
We now show that, under the right conditions, $\overline{\CH}$ is
quasi-abelian. We note that $\CH=\CC^{-}\iff\CC^{-}\subseteq\CC^{+}\iff\sus^{-1}\CS*\CW\subseteq\CW*\sus\CV$,
and dually that $\CH=\CC^{+}\iff\CC^{+}\subseteq\CC^{-}\iff\CW*\sus\CV\subseteq\sus^{-1}\CS*\CW$.
\begin{thm}
\label{thm:CC tri'd cat, CH =00003D CC^- or CC^+ then stable CH is quasi-abelian}Let
$\CC$ be a triangulated category with a twin cotorsion pair $((\CS,\CT),(\CU,\CV))$.
If $\CH=\CC^{-}$ or $\CH=\CC^{+}$, then $\overline{\CH}=\CH/[\CW]$
is quasi-abelian.\end{thm}
\begin{proof}
Since $\overline{\CH}$ is semi-abelian \cite[Thm. 5.4]{Nakaoka-twin-cotorsion-pairs},
we have that $\overline{\CH}$ is left quasi-abelian if and only if
$\overline{\CH}$ is right quasi-abelian \cite[Prop. 3]{Rump-almost-abelian-cats}.
Therefore, we will show that if $\CH=\CC^{-}$ then $\overline{\CH}$
is left quasi-abelian. Showing $\overline{\CH}$ is right quasi-abelian
whenever $\CH=\CC^{+}$ is similar.

Suppose $\CH=\CC^{-}$ and that we have a pullback diagram $$\begin{tikzcd}
A \arrow{r}{\overline{a}}\arrow{d}[swap]{\overline{b}}\PB{dr}& B\arrow{d}{\overline{c}} \\
C \arrow{r}[swap]{\overline{d}}& D
\end{tikzcd}$$in $\overline{\CH}$, where $\overline{d}$ is a cokernel. We need to show that $\ol{a}$ is also a cokernel.

By \cite[Lem. 5.1]{Nakaoka-twin-cotorsion-pairs},
we may assume that $d\in\Hom_{\CH}(C,D)$ is a morphism for which
there is a distinguished triangle $\begin{tikzcd}[column sep=1.2em]\sus^{-1}S\arrow{r}{}& C\arrow{r}{d}& D\arrow{r}{e}& S\end{tikzcd}$
in $\CC$ with $S\in\CS$. An application of (TR3) yields a triangle $\begin{tikzcd}[column sep=1.2em] C\arrow{r}{d}& D\arrow{r}{e}& S\arrow{r}{}&\sus C\end{tikzcd}$.

We can complete $c$ to a triangle $\begin{tikzcd}[column sep=1.2em]\sus^{-1}E\arrow{r}{}& B\arrow{r}{c}& D\arrow{r}{f}&E\end{tikzcd}$
and complete the composition $fd\colon C\to E$ to another
triangle $\begin{tikzcd}C\arrow{r}{fd}&E\arrow{r}&\sus X\arrow{r}{\sus x_{C}}&\sus C.\end{tikzcd}$
We obtain a triangle $\begin{tikzcd}D\arrow{r}{f}&E\arrow{r}&\sus B\arrow{r}{-\sus c}&\sus D\end{tikzcd}$ using (TR3). Then by the octahedral axiom (TR5), there is a commutative
diagram \begin{equation}\label{eqn:octahedral-on-fd}
\begin{tikzcd}[column sep = 1.3cm]
C \arrow{r}{d}\arrow[equals]{d}& D\arrow{r}{e}\arrow{d}{f} & S\arrow{r}\arrow{d}[swap]{-\sus g} & \sus C\arrow[equals]{d}\\
C \arrow{r}{fd}\arrow{d}[swap]{d}& E\arrow{r}\arrow[equals]{d} & \sus X \arrow{r}{\sus x_{C}}\arrow{d}[swap]{-\sus x_{B}}& \sus C \arrow{d}{\sus d}\\
D \arrow{r}{f}\arrow{d}[swap]{e}& E \arrow{r}\arrow{d}& \sus B \arrow{r}{-\sus c}\arrow[equals]{d}& \sus D \arrow{d}{\sus e}\\
S \arrow{r}[swap]{-\sus g}& \sus X \arrow{r}[swap]{-\sus x_{B}}& \sus B\arrow{r}[swap]{-\sus (ec)} & \sus S
\end{tikzcd}
\tag{$*$}
\end{equation}
in $\CC$ where the rows are triangles. There is a triangle $\begin{tikzcd}[column sep=1.6em]\sus^{-1}S\arrow{r}{g}&X\arrow{r}{x_{B}}&B\arrow{r}{ec}&S\end{tikzcd}$ using (TR3). 
Since $S\in\CS$ and $B\in\CH=\CC^{-}$, we have $X\in\CC^{-}=\CH$ by \cite[Lem. 2.12 (2)]{Nakaoka-twin-cotorsion-pairs}.

Since $B\in\CH\subseteq\CC^{+}=\CW*\sus\CV$, there is a triangle
$\begin{tikzcd}[column sep=1.3em] W\arrow{r}{h}& B\arrow{r}{i}& \sus V\arrow{r}{}& \sus W\end{tikzcd}$
with $W\in\CW$ and $V\in\CV$. Consider the composition $(ec)\circ h\colon W\to S$
and apply the octahedral axiom (TR5) to get the following commutative diagram
\begin{equation}\label{eqn:octahedral-on-(ec)h}
\begin{tikzcd}[column sep=1.3cm]
W \arrow{r}{h}\arrow[equals]{d}& B\arrow{r}{i}\arrow{d}{ec}  & \sus V\arrow{r}\arrow{d}{-\sus l} & \sus W\arrow[equals]{d}\\
W \arrow{r}{ech}\arrow{d}[swap]{h}&S\arrow{r}{-\sus j}\arrow[equals]{d}& \sus Y \arrow{r}{-\sus k}\arrow{d}{\sus y}& \sus W \arrow{d}{\sus h}\\
B \arrow{r}{ec}\arrow{d}[swap]{i}& S \arrow{r}{-\sus g}\arrow{d}{-\sus j} & \sus X \arrow{r}{-\sus x_{B}}\arrow[equals]{d}& \sus B \arrow{d}{\sus i}\\
\sus V \arrow{r}[swap]{-\sus l}      & \sus Y \arrow{r}[swap]{\sus y}& \sus X\arrow{r}[swap]{-\sus (ix_{B})} & \sus^{2}V
\end{tikzcd}
\tag{$**$}
\end{equation}
in $\CC$ with rows as triangles. Then we see that $\begin{tikzcd}[column sep=1cm]Y\arrow{r}{-y}&X\arrow{r}{ix_{B}}&\sus V\arrow{r}{-\sus l}&\sus Y\end{tikzcd}$
is also a triangle using (TR3) on the bottom triangle of \eqref{eqn:octahedral-on-(ec)h}. Moreover, this implies that $\begin{tikzcd}[column sep=0.8cm]Y\arrow{r}{y}&X\arrow{r}{ix_{B}}&\sus V\arrow{r}{\sus l}&\sus Y\end{tikzcd}$
is a triangle in $\CC$, using the triangle isomorphism $(-1_{Y},1_{X},1_{\sus V})$.

We claim that $\overline{x_{B}}\colon X\to B$ is a cokernel for $\overline{y}\colon Y\to X$
in $\overline{\CH}$. The morphism $x_{B}\in\Hom_{\CH}(X,B)$ embeds
in the triangle $\begin{tikzcd}X\arrow{r}{x_{B}}&B\arrow{r}{ec}&S\arrow{r}{-\sus g}&\sus X\end{tikzcd}$
where $S\in\CS\subseteq\CU$, and hence $\overline{x_{B}}$ is epic
in $\overline{\CH}$ by Proposition \ref{prop:twin cotorsion pair, f in CH(A,B) is epic in quotient iff Z_M_f in CW iff M_f in CU iff next map in triangle for f factors through CU}
as $ec$ factors through $\CU$. Thus, by Lemma \ref{lem:(epi) mono weak (co)kernel is (co)kernel}
it suffices to show that $\overline{x_{B}}$ is a weak cokernel for
$\overline{y}$. First, from (\ref{eqn:octahedral-on-(ec)h}) we see
that $x_{B}y=hk$ factors through $\CW$ as $W\in\CW$. Thus, in the
factor category $\overline{\CH}$ we have $\overline{x_{B}y}=0$.
Now suppose that there is some $m\colon X\to M$ such that $\overline{my}=0$
in $\overline{\CH}$. Then $my$ factors through $\CW$, so there
is a commutative diagram $$\begin{tikzcd} Y \arrow{r}{y}\arrow{d}[swap]{n} \commutes{dr}& X\arrow{d}{m} \\ W_{0} \arrow{r}[swap]{q}& M \end{tikzcd}$$
in $\CC$, where $W_{0}\in\CW$. Thus, we have a morphism of triangles $$\begin{tikzcd}Y\arrow{r}{y}\arrow{d}{n} & X\arrow{r}{ix_{B}}\arrow{d}{m}& \sus V\arrow{r}{\sus l}\arrow[dotted]{d}{\exists p} &\sus Y\arrow{d}{\sus n}\\ 
W_{0} \arrow{r}{q}& M \arrow{r}&N\arrow{r}{r}&\sus W_{0}\end{tikzcd}$$where $p$ exists using (TR4). From the commutativity of (\ref{eqn:octahedral-on-(ec)h}),
we have another morphism $$\begin{tikzcd} \sus ^{-1}S \arrow{r}{g}\arrow{d}{j}& X\arrow{r}{x_{B}}\arrow[equals]{d} & B\arrow{r}{ec}\arrow{d}{i} & S\arrow{d}{\sus j} \\
Y \arrow{r}{y}& X\arrow{r}{ix_{B}} & \sus V\arrow{r}{\sus l}& \sus Y \end{tikzcd}$$of triangles, where $(-\sus l)\circ i=(-\sus j)\circ ec$ and $(\sus y)\circ(-\sus j)=-\sus g$
yield $(\sus l)\circ i=(\sus j)\circ ec$ and $yj=g$, respectively.
Therefore, composing these two morphisms of triangles we get a commutative
diagram $$\begin{tikzcd}
\sus ^{-1}S \arrow{r}{g}\arrow{d}{nj}& X\arrow{r}{x_{B}}\arrow{d}{m} & B\arrow{r}{ec}\arrow{d}{pi} & S\arrow{d}{\sus(nj)} \\
W_{0} \arrow{r}{q}& M \arrow{r}&N\arrow{r}{r}&\sus W_{0}
\end{tikzcd}$$where the two rows are triangles. Notice that $\sus(nj)\in\Hom_{\CC}(S,\sus W_{0})=\Ext_{\CC}^{1}(S,W_{0})=0$
as $W_{0}\in\CW\subseteq\CT=\tensor[^{\perp_{1}}]{\CS}{}$. This implies
$r\circ pi$ vanishes, so there exists $\phi_{1}\colon X\to W_{0}$
and $\phi_{2}\colon B\to M$ such that $m=q\phi_{1}+\phi_{2}x_{B}$
by \cite[Lem. 3.2]{BuanMarsh-BM2}. Finally, in $\overline{\CH}$
we have $\overline{m}=\overline{q\phi_{1}}+\overline{\phi_{2}x_{B}}=\overline{\phi_{2}x_{B}}$
as $W_{0}\in\CW$ so $\overline{q\phi_{1}}=0$. This says that $\overline{x_{B}}$
is indeed a weak cokernel, and hence a cokernel, of $\overline{y}$.

In (\ref{eqn:octahedral-on-fd}) we see $(-\sus c)(-\sus x_{B})=(\sus d)(\sus x_{C}),$
so $cx_{B}=dx_{C}$ in $\CH$ and $\overline{cx_{B}}=\overline{dx_{C}}$
in the (left) semi-abelian category $\overline{\CH}$. Therefore,
since $\overline{x_{B}}$ is a cokernel, it follows from Lemma \ref{lem:CA left semi-abelian category pullback square A B C D, cokernel map X to B implies A to B is a cokernel}
that $\overline{a}\colon A\to B$ must also be a cokernel.

Hence, $\ol{\CH}$ is left quasi-abelian and thus quasi-abelian (by \cite[Prop. 3]{Rump-almost-abelian-cats}).\end{proof}
\begin{cor}
\label{cor:CC tri'd, twin cotorsion pair, if CU in CT or CT in CU then stable heart is both integral and quasi-abelian}Let
$\CC$ be a triangulated category with a twin cotorsion pair $((\CS,\CT),(\CU,\CV))$.
If $\CU\subseteq\CT$ or $\CT\subseteq\CU$, then $\overline{\CH}=\CH/[\CW]$
is quasi-abelian.\end{cor}
\begin{proof}
If $\CU\subseteq\CT$ then $\CH=\CC^{-}$. Therefore, we may apply Theorem
\ref{thm:CC tri'd cat, CH =00003D CC^- or CC^+ then stable CH is quasi-abelian}
to get that $\overline{\CH}$ is quasi-abelian.
\end{proof}
Note that in this case $\overline{\CH}$ is also integral: this
follows from \cite[Thm. 6.3]{Nakaoka-twin-cotorsion-pairs}
since $\CU\subseteq\CT$ implies $\CU\subseteq\CS*\CT$. We also remark
that in \cite{LiuYNakaoka-hearts-of-twin-cotorsion-pairs-on-extriangulated-categories}
there is the corresponding result for exact categories.

\section{\label{sec:Localisation of integral heart}Localisation of an integral
heart of a twin cotorsion pair}

For this section, we fix a Krull-Schmidt, triangulated category $\CC$
and a twin cotorsion pair $((\CS,\CT),(\CU,\CV))$ on $\CC$ with
$\CT=\CU$. In this setting, we have that the heart of $(\CS,\CT)$ is $\overline{\mathscr{H}}\deff\overline{\mathscr{H}}_{(\CS,\CT)}=(\sus^{-1}\CS*\CS)/[\CS]$
and the heart of $((\CS,\CT),(\CU,\CV))$ is $\overline{\CH}=\CC/[\CW]$,
where $\CW=\CT=\CU$ (see Definitions \ref{def:heart of individual cotorsion pair}
and \ref{def:quotients of C+ C- and H}, respectively). We will show
that there is an equivalence $\overline{\mathscr{H}}\simeq\overline{\CH}_{\CR}$
(see Theorem \ref{thm:CT=00003DCU implies G-Z localisation of CH at regulars is equivalent to heart of cotorsion pair (CS,CT)}),
where $\overline{\CH}_{\CR}$ is the (Gabriel-Zisman) localisation
of $\overline{\CH}$ at the class $\CR$ of regular morphisms in $\overline{\CH}$.

Our line of proof will be as follows: first, we obtain a canonical
functor $F\colon\overline{\mathscr{H}}\to\overline{\CH}$; then, composing
with the localisation functor $L_{\CR}\colon\overline{\CH}\to\overline{\CH}_{\CR}$,
we get a functor $\overline{\mathscr{H}}\to\overline{\CH}_{\CR}$
that we show is fully faithful and dense. The proofs in this section
are inspired by methods used in \cite{BuanMarsh-BM1},
\cite{BuanMarsh-BM2} and \cite{MarshPalu-nearly-morita-equivalences-and-rigid-objects}.

For the convenience of the reader, we recall some details of the description
of the localisation of an integral category at its regular morphisms.
To this end, suppose $\CA$ is an integral category and let $\CR_{\CA}$
be the class of regular morphisms in $\CA$. In this case, $\CR_{\CA}$
admits a \emph{calculus of left fractions }(see \cite[\S I.2]{GabrielZisman-calc-of-fractions})
by \cite[Prop. 6]{Rump-almost-abelian-cats}. The objects
of the localisation $\CA_{\CR_{\CA}}$ are the objects of $\CA$.
A morphism in $\CA_{\CR_{\CA}}$ from $X$ to $Y$ is a \emph{left
fraction }of the form $$\begin{tikzcd}[row sep=0.3cm, column sep=1.5cm]
&A&\\
X\arrow{ur}{f}&&Y\arrow{ul}[swap]{r}
\end{tikzcd}$$denoted $[f,r]_{\LF}$, up to a certain equivalence (see \cite[\S I.2]{GabrielZisman-calc-of-fractions}
for more details), where $f$ is any morphism in $\CA$ and $r$ is
in $\CR_{\CA}$. The localisation functor $L_{\CR_{\CA}}\colon\CA\to\CA_{\CR_{\CA}}$
is the identity on objects and takes a morphism $f\colon X\to A$
to the left fraction $L_{\CR_{\CA}}(f)=[f]\deff[f,1_{A}]_{\LF}$.
If $r\colon Y\to A$ is in $\CR_{\CA}$, then the morphism $[r]$
in $\CA_{\CR_{\CA}}$ is invertible with inverse $[r]^{-1}$ equal to the left fraction $[1_{A},r]_{\LF}$. An exposition of the morphisms 
as right fractions may be found in \cite[\S4]{BuanMarsh-BM2}.

We note that the localisation as described above may not exist without
passing to a higher universe. However, as in \cite{MarshPalu-nearly-morita-equivalences-and-rigid-objects},
we will show that the localisations considered in the remainder of
this article are equivalent to certain subfactors of locally small
categories, and hence locally small themselves. Thus, the localisations
we are interested in are already categories and we need not pass to
a higher universe.
\begin{prop}
\label{prop:canonical functor from H (heart of (CS,CT)) to heart CH of twin cotorsion pair} 
There is an additive functor $F\colon\overline{\mathscr{H}}\to\overline{\CH}$
that is the identity on objects and, for a morphism $f\colon X\to Y$
in $\sus^{-1}\CS*\CS$, maps the coset $f+[\CS](X,Y)$ to the coset $f+[\CW](X,Y)$.\end{prop}
\begin{proof}
The full subcategory $\mathscr{H}=\sus^{-1}\CS*\CS\subseteq\CC=\CH$
comes equipped with an inclusion functor $\iota\colon\mathscr{H}\to\CH$, which may be composed
with the additive quotient functor $Q_{[\CW]}\colon\CH\to\CH/[\CW]$
to get a functor $Q_{[\CW]}\circ\iota\colon\mathscr{H}\to\overline{\CH}$.
Note that this functor maps any morphism in the ideal $[\CS]$ to
$0$ as $\CS\subseteq\CU=\CW$, and therefore we get the following
commutative diagram $$\begin{tikzcd}[row sep=1.2cm,column sep=1.2cm]
\mathscr{H}=\sus^{-1}\CS * \CS \arrow{r}{\iota}\arrow{d}[swap]{Q_{[\CS]}}\commutes{dr}& \CC=\CH \arrow{d}{Q_{[\CW]}} \\
\overline{\mathscr{H}}=(\sus^{-1}\CS * \CS)/[\CS] \arrow[dotted]{r}[swap]{\exists ! F}& \CC /[\CW]=\overline{\CH}
\end{tikzcd}$$ of additive categories using the universal property of the additive
quotient $\overline{\mathscr{H}}$. Furthermore, we see that $F(X)=F(Q_{[\CS]}(X))=Q_{[\CW]}(\iota(X))=X$
and 
\[
F(f+[\CS](X,Y))=F(Q_{[\CS]}(f))=Q_{[\CW]}(\iota(f))=Q_{[\CW]}(f)=f+[\CW](X,Y).
\]

\end{proof}
The next result below is a characterisation of the regular morphisms
in $\overline{\CH}=\CC/[\CW]$, and is a special case of \cite[Lem. 4.1]{Beligiannis-rigid-objects-triangulated-subfactors-and-abelian-localizations}.
Note that $\sus^{-1}\CS$ is a contravariantly finite and \emph{rigid}
(i.e. $\Ext_{\CC}^{1}(\sus^{-1}\CS,\sus^{-1}\CS)=0$) subcategory
of $\CC$, because $\CS\subseteq\CU=\CT=\CS^{\perp_{1}}$, and that
$(\sus^{-1}\CS)^{\perp_{0}}=\CS^{\perp_{1}}=\CT=\CU=\CW$.

\begin{prop}
\label{prop:monicity, epicity, regularity mod CW}Suppose $\sus^{-1}Z\overset{h}{\longrightarrow}X\overset{f}{\longrightarrow}Y\overset{g}{\longrightarrow}Z$
is a triangle in $\CC$. Denote by $\overline{f}$ the morphism $f+[\CW](X,Y)\in\Hom_{\overline{\CH}}(X,Y)$
in $\overline{\CH}$. \emph{\begin{enumerate}[(i)]
\item \emph{The morphism} $\overline{f}$ \emph{is monic if and only if} $h$ \emph{factors through} $\CW$\emph{.}
\item \emph{The morphism} $\overline{f}$ \emph{is epic if and only if} $g$ \emph{factors through} $\CW$\emph{.}
\item \emph{The morphism} $\overline{f}$ \emph{is regular if and only if} $h$ \emph{and} $g$ \emph{factor through} $\CW$\emph{.}
\end{enumerate}}
\end{prop}
The following lemma is a generalisation of \cite[Lem. 3.3]{BuanMarsh-BM1}.
The proof of Buan and Marsh easily generalises, so we omit the proof
of our statement. One is able to recover the result of Buan and Marsh
by putting the appropriate restrictions on $\CC$ and by setting $((\CS,\CT),(\CU,\CV))=((\add\sus T,\CX_{T}),(\CX_{T},\tensor[]{\XT}{^{\perp_{1}}}))$, where $T$ is a rigid object in $\CC$.
\begin{lem}
\label{lem:every object of CC is isomorphic in localised CH to an object lying in sus-1CS * CS}Let
$Y$ be an arbitrary object of $\CC$. Then there exists $X\in\sus^{-1}\CS*\CS=\mathscr{H}$
and a morphism $\overline{r}\colon X\to Y$ in the class $\CR$ of regular morphisms in $\ol{\CH}$.
\end{lem}
By Proposition \ref{prop:canonical functor from H (heart of (CS,CT)) to heart CH of twin cotorsion pair},
we have an additive functor $F\colon\overline{\mathscr{H}}\to\overline{\CH}$.
Define $G\deff L_{\CR}\circ F$, where $L_{\CR}\colon\overline{\CH}\to\overline{\CH}_{\CR}$
is the additive localisation functor \cite[Rem. 4.3]{BuanMarsh-BM2},
to obtain the following commutative diagram of additive functors$$\begin{tikzcd} \overline{\mathscr{H}} \arrow{r}{F}\arrow{dr}[swap]{G}& \overline{\CH}\arrow{d}{L_{\CR}}\\
& \overline{\CH}_{\CR} \end{tikzcd}$$Note that $G(X)=X$ and $G(f+[\CS](X,Y))=[\overline{f}]=[f+[\CW](X,Y),1_{Y}]_{\LF}$.
The remainder of this section is dedicated to showing that $G$ is an equivalence
of categories.
\begin{prop}
\label{prop:G=00003DLF is dense}The functor $G\colon\overline{\mathscr{H}}\to\overline{\CH}_{\CR}$
is dense.\end{prop}
\begin{proof}
Recall that the objects of $\overline{\CH}_{\CR}$ are the objects
of $\CH=\CC$. Let $Y\in\overline{\CH}_{\CR}$ be arbitrary. Then
by Lemma \ref{lem:every object of CC is isomorphic in localised CH to an object lying in sus-1CS * CS},
there exists a morphism $r\colon X\to Y$ in $\CC$ with $X\in\sus^{-1}\CS*\CS=\mathscr{H}$,
such that $\overline{r}$ is regular in $\overline{\CH}$. Hence,
in $\overline{\CH}_{\CR}$ we have that $L_{\CR}(\overline{r})\colon X\to Y$
is an isomorphism so that $Y\iso X=L_{\CR}F(X)=G(X)$, and $G$ is
a dense functor.
\end{proof}
To show $G$ is faithful we need the following observation due to Beligiannis.
\begin{lem}
\label{lem:X in CS=00005B-1=00005D*CS, then f X to Y factors through CW implies f factors through CS}\emph{\cite[Rem. 4.3 (iii)]{Beligiannis-rigid-objects-triangulated-subfactors-and-abelian-localizations}}
Suppose $X\in\sus^{-1}\CS*\CS$ and $f\colon X\to Y$ is a morphism
in $\CC$. If $f$ factors through $\CW$, then $f$ factors through
$\CS$.
\end{lem}

\begin{prop}
\label{prop:G=00003DLF is faithful}The functor $G\colon\overline{\mathscr{H}}\to\overline{\CH}_{\CR}$
is faithful.\end{prop}
\begin{proof}
Suppose $\overline{f}=f+[\CS](X,Y)\colon X\to Y$ is a morphism in
$\overline{\mathscr{H}}=(\sus^{-1}\CS*\CS)/[\CS]$ such that $G(\overline{f})=L_{\CR}(F(\overline{f}))=0$
in $\overline{\CH}_{\CR}=(\CC/[\CW])_{\CR}$. Then $f+[\CW](X,Y)=F(\overline{f})=0$
in $\overline{\CH}=\CC/[\CW]$ because $L_{\CR}\colon\overline{\CH}\to\overline{\CH}_{\CR}$
is faithful by \cite[Lem. 4.4]{BuanMarsh-BM2}. Hence,
$f$ factors through $\CW$ in $\CC$. Note that $X\in\sus^{-1}\CS*\CS$,
so $f$ factors through $\CS$ by Lemma \ref{lem:X in CS=00005B-1=00005D*CS, then f X to Y factors through CW implies f factors through CS}
and $\overline{f}$ is the zero morphism in $\overline{\mathscr{H}}$.
Therefore, the functor $G$ is faithful.
\end{proof}

\begin{prop}
\label{prop:G=00003DLF is full}The functor $G\colon\overline{\mathscr{H}}\to\overline{\CH}_{\CR}$
is full.\end{prop}
\begin{proof}
Let $X,Y$ be objects in $\overline{\mathscr{H}}$ and consider the mapping
$$\Hom_{\overline{\mathscr{H}}}(X,Y)\longrightarrow\Hom_{\overline{\CH}_{\CR}}(G(X),G(Y))=\Hom_{\overline{\CH}_{\CR}}(X,Y).$$
Let $\begin{tikzcd}[column sep=0.75cm]
X\arrow{r}{\overline{f}}&A&Y\arrow{l}[swap]{\overline{r}}
\end{tikzcd}$be an arbitrary morphism in $\Hom_{\overline{\CH}_{\CR}}(G(X),G(Y))$.
Since $r\colon Y\to A$ is a morphism in $\CC$ such that $\overline{r}$
is regular in $\overline{\CH}$, there is a triangle $\begin{tikzcd}[column sep=1em] \sus^{-1}Z \arrow{r}{s}& Y\arrow{r}{r}& A\arrow{r}{t}& Z\end{tikzcd}$
such that $s,t$ factor through $\CW$ by Proposition \ref{prop:monicity, epicity, regularity mod CW}. As $X\in\obj(\overline{\mathscr{H}})=\obj(\mathscr{H})=\obj(\sus^{-1}\CS*\CS)$,
there exists a triangle $\begin{tikzcd}[column sep=1.5em] \sus^{-1}S_1 \arrow{r}{a}&\sus^{-1}S_0 \arrow{r}{b}&X \arrow{r}{c}&S_1 \end{tikzcd}$
in $\CC$ with $S_{0},S_{1}\in\CS$. Suppose $t\colon A\to Z$ factors
as $ed$ for some $d\colon A\to T$, $e\colon T\to Z$ with $T\in\CW=\CT$.
Then the morphism $dfb\in\Hom_{\CC}(\sus^{-1}S_{0},T)\iso\Ext_{\CC}^{1}(S_{0},T)=0$
vanishes, and hence $tfb=edfb$ is the zero map too. Thus, there exists
$g\colon\sus^{-1}S_{0}\to Y$ such that $rg=fb$. Applying (TR4),
we obtain a morphism $$\begin{tikzcd}
\sus^{-1}S_1 \arrow{r}{a}\arrow[dotted]{d}{h}&\sus^{-1}S_0 \arrow{r}{b}\arrow{d}{g}&X \arrow{r}{c}\arrow{d}{f}&S_1\arrow[dotted]{d}{\sus h} \\
\sus^{-1}Z \arrow{r}{s}&Y\arrow{r}{r}&\arrow{r}{}A\arrow{r}{t}&Z
\end{tikzcd}$$of triangles in $\CC$, in which $ga=sh$ vanishes as $S_{1}\in\CS$ and $s$
factors through $\CW=\CT$. Hence, by \cite[Lem. 3.2]{BuanMarsh-BM2},
there are morphisms $u\in\Hom_{\CC}(X,Y)=\Hom_{\mathscr{H}}(X,Y)$
and $v\in\Hom_{\CC}(S_{1},A)$ such that $f=ru+vc$ in $\CC$. Therefore,
in $\overline{\CH}=\CC/[\CW]$ we have $\overline{f}=\overline{ru}+\overline{vc}=\overline{ru}$
as $\overline{v}=0$ because $S_{1}\in\CS\subseteq\CU=\CW$. This
implies that $[\overline{f}]=[\overline{ru}]=[\overline{r}][\overline{u}]$,
and hence $[u+[\CW](X,Y)]=[\overline{u}]=[\overline{r}]^{-1}[\overline{f}]=[f,r]_{\LF}$
in $\overline{\CH}_{\CR}$. Finally, we see that 
\[
[f,r]_{\LF}=[u+[\CW](X,Y)]=[F(u+[\CS](X,Y))]=L_{\CR}F(u+[\CS](X,Y))=G(u+[\CS](X,Y)),
\]
and the map $\Hom_{\overline{\mathscr{H}}}(X,Y)\to\Hom_{\overline{\CH}_{\CR}}(G(X),G(Y))$
is surjective, i.e. $G$ is a full functor. \phantom{to get qed box on new line}\end{proof}
\begin{thm}
\label{thm:CT=00003DCU implies G-Z localisation of CH at regulars is equivalent to heart of cotorsion pair (CS,CT)}Let
$\CC$ be a Krull-Schmidt, triangulated category. Suppose $((\CS,\CT),(\CU,\CV))$
is a twin cotorsion pair on $\CC$ that satisfies $\CT=\CU$. Let
$\CR$ denote the class of regular morphisms in the heart $\overline{\CH}$
of $((\CS,\CT),(\CU,\CV))$. Then the Gabriel-Zisman localisation
$\overline{\CH}_{\CR}$ is equivalent to the heart $\overline{\mathscr{H}}_{(\CS,\CT)}$
of the cotorsion pair $(\CS,\CT)$.\end{thm}
\begin{proof}
It is well-known that a fully faithful, dense functor is an equivalence. Hence, the functor $G=L_{\CR}\circ F$ gives an equivalence $\overline{\mathscr{H}}_{(\CS,\CT)}\overset{\simeq}{\longrightarrow}\overline{\CH}_{\CR}$
using Propositions \ref{prop:G=00003DLF is dense}, \ref{prop:G=00003DLF is faithful}
and \ref{prop:G=00003DLF is full} above.\end{proof}
\begin{rem}
Notice that although Theorem \ref{thm:CT=00003DCU implies G-Z localisation of CH at regulars is equivalent to heart of cotorsion pair (CS,CT)}
looks somewhat independent of the cotorsion pair $(\CU,\CV)$, we
have that the pair $(\CS,\CT)$ determines $(\CU,\CV)$, and vice
versa, using Proposition \ref{prop:elementary properties for cotorsion pair}
and that $\CT=\CU$.
\end{rem}

\begin{rem}
\label{rem:comparison to Bel13}We show now that the conclusion of
Theorem \ref{thm:CT=00003DCU implies G-Z localisation of CH at regulars is equivalent to heart of cotorsion pair (CS,CT)}
also follows from results of Beligiannis. Let $\CC$ be a category
as in the statement of Theorem \ref{thm:CT=00003DCU implies G-Z localisation of CH at regulars is equivalent to heart of cotorsion pair (CS,CT)}
above. Let $\CX$ be a contravariantly finite and rigid\emph{ }subcategory
of $\CC$. Suppose further that $\CX^{\perp_{0}}$ is contravariantly
finite. Beligiannis shows (see Remark 4.3, Lemma 4.4 and Theorem 4.6
in \cite{Beligiannis-rigid-objects-triangulated-subfactors-and-abelian-localizations})
that there are equivalences 
\[
(\CX*\sus\CX)/[\sus\CX]=(\CX*\sus\CX)/[\CX^{\perp_{0}}]\overset{\simeq}{\longrightarrow}\mod\CX\overset{\simeq}{\longleftarrow}(\CC/[\CX^{\perp_{0}}])_{\CR},
\]
where $\mod\CX$ is the category of coherent functors over $\CX$
(see \cite{Auslander-coherent-functors}) and $\CR$ is the class
of regular morphisms in the category $\CC/[\CX^{\perp_{0}}]$. In
the situation of Theorem \ref{thm:CT=00003DCU implies G-Z localisation of CH at regulars is equivalent to heart of cotorsion pair (CS,CT)},
we have that $\sus^{-1}\CS$ is a contravariantly finite, rigid subcategory,
and that $(\sus^{-1}\CS)^{\perp_{0}}=\CW=\CU$ is also contravariantly
finite; see the discussion above Proposition \ref{prop:monicity, epicity, regularity mod CW}
for more details. Therefore, with $\CX=\sus^{-1}\CS$ one obtains
\[
(\CX*\sus\CX)/[\sus\CX]=(\sus^{-1}\CS*\CS)/[\CS]=\overline{\mathscr{H}}_{(\CS,\CT)}
\]
and 
\[
(\CC/[\CX^{\perp_{0}}])_{\CR}=(\CC/[\CW])_{\CR}=\overline{\CH}_{\CR}.
\]
Hence, one may deduce that $\overline{\mathscr{H}}_{(\CS,\CT)}$ and
$\overline{\CH}_{\CR}$ are equivalent from the results in \cite{Beligiannis-rigid-objects-triangulated-subfactors-and-abelian-localizations}.
However, the proof method is different: Beligiannis makes use of adjoint
functors and obtains a functor $(\CC/[\CX^{\perp_{0}}])_{\CR}\to(\CX*\sus\CX)/[\CX^{\perp_{0}}]$,
which is stated to be an equivalence, using the universal property
of the localisation $(\CC/[\CX^{\perp_{0}}])_{\CR}$; on the other
hand, we construct an explicit equivalence in the other direction.
\end{rem}

\section{\label{sec:application to cluster category}An application to the
cluster category}

In this section, we assume $k$ is a field and that, unless otherwise
stated, $\CC$ is a $\Hom$-finite, Krull-Schmidt, triangulated $k$-category
with a Serre functor $\nu$. As usual, we will denote the suspension
functor of $\CC$ by $\Sigma$. For the convenience of the reader,
we recall the definition of a Serre functor below.
\begin{defn}
\label{def:Serre functor}\cite[\S2.6]{Keller-cy-tri-cats}
A \emph{Serre functor} of a $\Hom$-finite, triangulated $k$-category
$\CC$ is a triangle autoequivalence $\nu\colon\CC\to\CC$ such that
for any $X,Y\in\CC$ we have
\[
\Hom_{\CC}(X,Y)\iso D\Hom_{\CC}(Y,\nu X),
\]
which is functorial in both arguments and where $D(-)\deff\Hom_{\mod k}(-,k)$. In this case, we say $\CC$ \emph{has Serre duality}.
\end{defn}
For the remainder of this section, we also assume that $R$ is a fixed
\emph{rigid} object of $\CC$ (that is, $\Ext_{\CC}^{1}(R,R)=0$).
For an object $X$ in $\CC$, we denote by $\add X$ the full, additive subcategory of $\CC$ consisting
of objects that are isomorphic to direct summands of finite direct
sums of copies of $X$, and by $\CX_{X}$ the full, additive subcategory
of $\CC$ that consists of objects $Y$ such that $\Hom_{\CC}(X,Y)=0$.
Hence, $\CX_{X}$ is equal to $(\add X)^{\perp_{0}}$, or $(\add\sus X)^{\perp_{1}}$
as in \cite{BuanMarsh-BM2}. The next proposition collects
some easily verifiable observations, some of which may be found in
\cite{BuanMarsh-BM1}.
\begin{prop}
\label{prop:R rigid implies addR and X_R extension-closed and closed under direct summands and isos and X_R=00003Dadd R=00005B1=00005D Ext-perp}For
any rigid object $R'\in\CC$, the subcategories \emph{$\add R'$}
and $\CX_{R'}$ are closed under isomorphisms and direct summands.
Moreover, these subcategories are also extension-closed.
\end{prop}
 
\begin{rem}
\label{rem:in a KS cat, approximation exists implies minimal approx exists}Since
$\CC$ is a Krull-Schmidt category, if an object $X$ of $\CC$ has
a right $\CX$-approximation, for some subcategory $\CX\subseteq\CC$,
then $X$ has a minimal right $\CX$-approximation (see Definition
\ref{def:right left minimal morphism}) by \cite[Cor. 1.4]{KrauseSaorin-minimal-approximations-of-modules}.
Dually, the existence of a left $\CX$-approximation implies the existence
of a minimal such one under our assumptions.
\end{rem}
The next result is stated in \cite{Nakaoka-twin-cotorsion-pairs},
but we include the details to illustrate where the various assumptions on $\CC$ are needed. See
also \cite{BuanMarsh-BM2}.
\begin{lem}
\emph{\label{lem:Nakaoka twin cotorsion pairs Example 2.10 (2) twin cot pair giving cluster cat mod X_R}\cite[Exam. 2.10 (2)]{Nakaoka-twin-cotorsion-pairs}}
The pair \emph{$((\add\sus R,\CX_{R}),(\XR,\tensor[]{\XR}{^{\perp_1}}))$}
is a twin cotorsion pair with heart $\overline{\CH}=\CC/[\CX_{R}]$.\end{lem}
\begin{proof}
First, we show that $(\add\sus R,\CX_{R})$ is a cotorsion pair on
$\CC$. Since $\CC$ is assumed to be $\Hom$-finite, we have that
$\add\sus R$ is contravariantly finite, so for any $X\in\CC$ there
exists a triangle $\begin{tikzcd}[column sep=1.3em] \sus R_0 \arrow{r}{f}& X\arrow{r}{}& Y \arrow{r}{}& \sus^2 R_0, \end{tikzcd}$ where
$f\colon\sus R_{0}\to X$ is a minimal right $\add\sus R$-approximation
of $X$ because $\CC$ is also Krull-Schmidt. Since $\add\sus R$
is extension-closed (see Proposition \ref{prop:R rigid implies addR and X_R extension-closed and closed under direct summands and isos and X_R=00003Dadd R=00005B1=00005D Ext-perp}),
by Lemma \ref{lem:triangulated wakamatsu} we have that $Y\in(\add\sus R)^{\perp_{0}}=\sus\CX_{R}$.
Therefore, $\CC=\add\sus R*\sus\CX_{R}$. We also have $\Ext_{\CC}^{1}(\sus R,\CX_{R})=\Hom_{\CC}(\sus R,\sus\CX_{R})\iso\Hom_{\CC}(R,\CX_{R})=0$.
Comparing with Definition \ref{def:cotorsion pair in tri'd cat},
we see that $(\CS,\CT)\deff(\add\sus R,\CX_{R})$ is indeed a cotorsion
pair.

To see that $(\CU,\CV)\deff(\CX_{R},\tensor[]{\XR}{^{\perp_{1}}})$
is a cotorsion pair, take a minimal left $\add\nu R$-approximation
$r\colon X\to\nu R_{1}$ of $X$ and complete it to a triangle $\begin{tikzcd}[column sep=1.3em]Z \arrow{r}{s}& X\arrow{r}{r} & \nu R_1\arrow{r}{} & \sus Z.\end{tikzcd}$
Then by Lemma \ref{lem:triangulated wakamatsu} again, we have $Z\in\CX_{R}$
and so $\CC=\CX_{R}*\sus(\sus^{-1}\add\nu R)$. In addition, 
\[
\Ext_{\CC}^{1}(\CX_{R},\sus^{-1}\add\nu R)=\Hom_{\CC}(\CX_{R},\add\nu R)\iso\Hom_{\CC}(\add R,\CX_{R})=0,
\]
so $(\CX_{R},\add\sus^{-1}\nu R)$ is a cotorsion pair. Therefore,
by Proposition \ref{prop:elementary properties for cotorsion pair},
we see that $(\CU,\CV)=(\CX_{R},\tensor[]{\XR}{^{\perp_{1}}})=(\CX_{R},\add\sus^{-1}\nu R)$
is a cotorsion pair.

Furthermore, we have $\Hom_{\CC}(R,\sus R)=\Ext_{\CC}^{1}(R,R)=0$
as $R$ is rigid, so $\CS=\add\sus R\subseteq\CX_{R}=\CU$. Hence,
$((\CS,\CT),(\CU,\CV))=((\add\sus R,\CX_{R}),(\XR,\tensor[]{\XR}{^{\perp_1}}))$
is a twin cotorsion pair on $\CC$ (see Definition \ref{def:twin cotorsion pair in tri'd cat}).
In particular, $\CT=\CX_{R}=\CU$, so $\CW=\CT=\CU=\CX_{R}$, and 
\[
\CC^{-}=\sus^{-1}\CS*\CW=\sus^{-1}\CS*\CT=\CC=\CU*\sus\CV=\CW*\sus\CV=\CC^{+}.
\]
Therefore, $\CH=\CC^{-}\cap\CC^{+}=\CC$ and the heart associated
to $((\CS,\CT),(\CU,\CV))$ is $\overline{\CH}=\CH/[\CW]=\CC/[\CX_{R}]$.\end{proof}
\begin{thm}
Suppose $\CC$ is a $\Hom$-finite, Krull-Schmidt, triangulated $k$-category that has 
Serre duality, and assume $R$ is a rigid object of $\CC$.
Then $\CC/[\CX_{R}]$ is quasi-abelian.\end{thm}
\begin{proof}
Consider the twin cotorsion pair $((\CS,\CT),(\CU,\CV))=((\add\sus R,\CX_{R}),(\XR,\tensor[]{\XR}{^{\perp_1}}))$.
As $\CT=\CU$ in this case, by Corollary \ref{cor:CC tri'd, twin cotorsion pair, if CU in CT or CT in CU then stable heart is both integral and quasi-abelian},
$\overline{\CH}=\CC/[\CX_{R}]$ is quasi-abelian.
\end{proof}

Let $\CR$ be the class of regular morphisms in $\CC/[\CX_{R}]$,
and denote by $\CC(R)$ the subcategory $(\add R)*(\add\sus R)$ considered
in \cite[\S5.1]{KellerReiten-ct-algebras-are-gorenstein-stably-cy};
see also \cite[Prop. 6.2]{IyamaYoshino-mutation-in-tri-cats-rigid-CM-mods},
\cite{BuanMarsh-BM1} and \cite{BuanMarsh-BM2}.
An equivalence between $\CC(R)/[\add\sus R]$ and $(\CC/[\CX_{R}])_{\CR}$
exists by combining \cite[Prop. 6.2]{IyamaYoshino-mutation-in-tri-cats-rigid-CM-mods}
with \cite[Thm. 5.7]{BuanMarsh-BM2} (or results of \cite{Beligiannis-rigid-objects-triangulated-subfactors-and-abelian-localizations}
as discussed in Remark \ref{rem:comparison to Bel13}) as follows
\[
\CC(R)/[\add\sus R]\overset{\simeq}{\longrightarrow}\mod\Lambda_{R}\overset{\simeq}{\longleftarrow}(\CC/[\CX_{R}])_{\CR},
\]
where $\Lambda_{R}\deff(\End_{\CC}R)^{\op}$. We now give a new proof
that $\CC(R)/[\add\sus R]$ and $(\CC/[\CX_{R}])_{\CR}$ are equivalent,
which avoids going via the module category $\mod\Lambda_{R}$ altogether. 
\begin{thm}
Let $\CC$ be a $\Hom$-finite, Krull-Schmidt, triangulated $k$-category,
and assume $R$ is a rigid object of $\CC$. Let $\CR$ be the class of regular morphisms in $\CC/[\CX_{R}]$ and let $L_{\CR}\colon\CC/[\CX_{R}]\to(\CC/[\CX_{R}])_{\CR}$
be the localisation functor. Then there is an additive functor \emph{$F\colon\CC(R)/[\add\sus R]\to\CC/[\CX_{R}]$}
such that the composition\emph{
\[
L_{\CR}\circ F\colon\CC(R)/[\add\sus R]\overset{\simeq}{\longrightarrow}(\CC/[\CX_{R}])_{\CR}
\]
}is an equivalence.\end{thm}
\begin{proof}
Let $((\CS,\CT),(\CU,\CV))=((\add\sus R,\CX_{R}),(\XR,\tensor[]{\XR}{^{\perp_1}}))$.
The heart (see Definition \ref{def:heart of individual cotorsion pair})
of the cotorsion pair $(\CS,\CT)=(\add\sus R,\CX_{R})$ is $\overline{\mathscr{H}}_{(\CS,\CT)}=\CC(R)/[\add\sus R]$, 
and the heart of the twin cotorsion pair $((\CS,\CT),(\CU,\CV))$
is $\overline{\CH}=\CC/[\CX_{R}]$. By Proposition \ref{prop:canonical functor from H (heart of (CS,CT)) to heart CH of twin cotorsion pair}
there is an additive functor $F\colon\CC(R)/[\add\sus R]\to\CC/[\CX_{R}]$
that is the identity on objects and maps a morphism $f+[\add\sus R](X,Y)$
to $f+[\CX_{R}](X,Y)$, which is well-defined as $\add\sus R\subseteq\CX_{R}$
since $R$ is a rigid. Then an application of Theorem \ref{thm:CT=00003DCU implies G-Z localisation of CH at regulars is equivalent to heart of cotorsion pair (CS,CT)}
yields an equivalence $$\CC(R)/[\add\sus R]=\overline{\mathscr{H}}_{(\CS,\CT)}\overset{\simeq}{\longrightarrow}\overline{\CH}_{\CR}=(\CC/[\CX_{R}])_{\CR}.$$
\end{proof}
We make two last observations before giving an example to demonstrate
this theory.
\begin{defn}
\label{def:n CY}\cite[\S2.6]{Keller-cy-tri-cats} For $n\in\BN,$
we say that a $\Hom$-finite, triangulated $k$-category $\CC$ is 
\emph{$n$-Calabi-Yau} if $\CC$ admits a Serre functor $\nu$ such
that there is a natural isomorphism $\nu\iso\sus^{n}$ as $k$-linear
triangle functors.\end{defn}
\begin{prop}
\label{prop:CC tri'd, 2 CY, CU=00003DCT iff CS=00003DCV for twin cotorsion pair}Let
$\CC$ be a $\Hom$-finite, triangulated $k$-category which is 2-Calabi-Yau.
Suppose $(\CS,\CT)$ and $(\CU,\CV)$ are cotorsion pairs on $\CC$.
Then $\CT=\CU$ if and only if $\CS=\CV$.\end{prop}
\begin{proof}
Assume $\CT=\CU$. Then we have the following chain of equalities
$$\CS=\tensor[^{\perp_{1}}]{\CT}{}=\tensor[^{\perp_{1}}]{\CU}{}=\CU^{\perp_{1}}=\CV,$$using
Proposition \ref{prop:elementary properties for cotorsion pair} and
that $\CC$ is 2-Calabi-Yau. The converse is proved similarly.\end{proof}
\begin{cor}
\label{cor:C 2CY implies add susR equals right Ext-perp of X_R}Let
$\CC$ be a $\Hom$-finite, Krull-Schmidt, 2-Calabi-Yau, triangulated
$k$-category, and assume $R$ is a rigid object of $\CC$. Then the
subcategory \emph{$\add\sus R$} coincides with \emph{$\tensor[]{\XR}{^{\perp_1}}$}.\end{cor}
\begin{proof}
Since $((\CS,\CT),(\CU,\CV))=((\add\sus R,\CX_{R}),(\XR,\tensor[]{\XR}{^{\perp_1}}))$
is a pair of cotorsion pairs on $\CC$ with $\CT=\CU$, we must have $\add\sus R=\CS=\CV=\tensor[]{\XR}{^{\perp_{1}}}$
by Proposition \ref{prop:CC tri'd, 2 CY, CU=00003DCT iff CS=00003DCV for twin cotorsion pair}.
\end{proof}

It can also be shown that Corollary \ref{cor:C 2CY implies add susR equals right Ext-perp of X_R}
follows from \cite[Lem. 2.2]{BuanMarsh-BM1} using $T=\sus R$
and the 2-Calabi-Yau property.
\begin{example}
\label{exa:C of A_4 example}Consider the cluster category $\CC\deff\CC_{Q}$
associated to the linearly oriented Dynkin quiver 
\[
Q:\quad1\to2\to3\to4.
\]
Its Auslander-Reiten quiver, with the mesh relations omitted, is $$\makebox[\textwidth][c]{\includegraphics[width=0.8\textwidth]{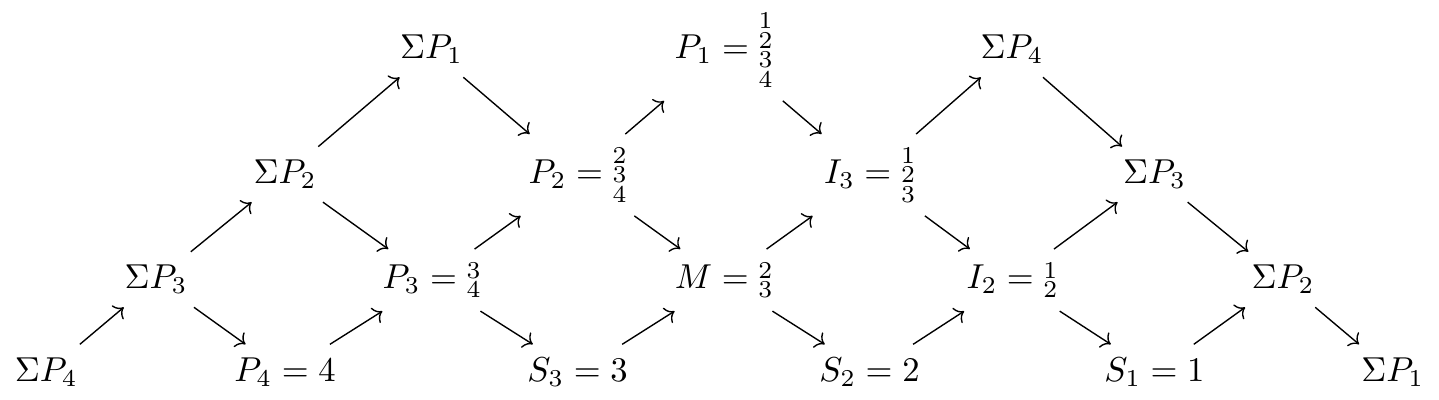}}$$where
the lefthand copy of $\sus P_{i}$ is identified with the corresponding righthand
copy (for $i=1,2,3,4$) (see, for example, \cite[\S3.1]{Schiffler-quiver-reps-book}). We set $R\deff P_{1}\oplus P_{2}\oplus S_{2}$,
which is a basic, rigid object of $\CC$. Note that since $R$ has
just $3$ non-isomorphic indecomposable direct summands, it is not
maximal rigid (see \cite[Cor. 2.3]{BuanMarshReiten-cluster-tilted-algebras}) and hence not cluster-tilting. Denote by $\Lambda_{R}$
the ring $(\End_{\CC}R)^{\op}$. We describe the twin cotorsion pair
$((\CS,\CT),(\CU,\CV))=((\add\sus R,\CX_{R}),(\XR,\tensor[]{\XR}{^{\perp_1}}))$
pictorially below, where ``$\circ$'' denotes that the corresponding
object does not belong to the subcategory. Since the cluster category
is 2-Calabi-Yau (see \cite{BMRRT-cluster-combinatorics}),
that $\CS$ coincides with $\CV$ below is not unexpected (see Corollary \ref{cor:C 2CY implies add susR equals right Ext-perp of X_R}).

\begin{gather*}
\CT = \CU= \CW=\CX_R \\ 
\includegraphics[width=0.8\textwidth]{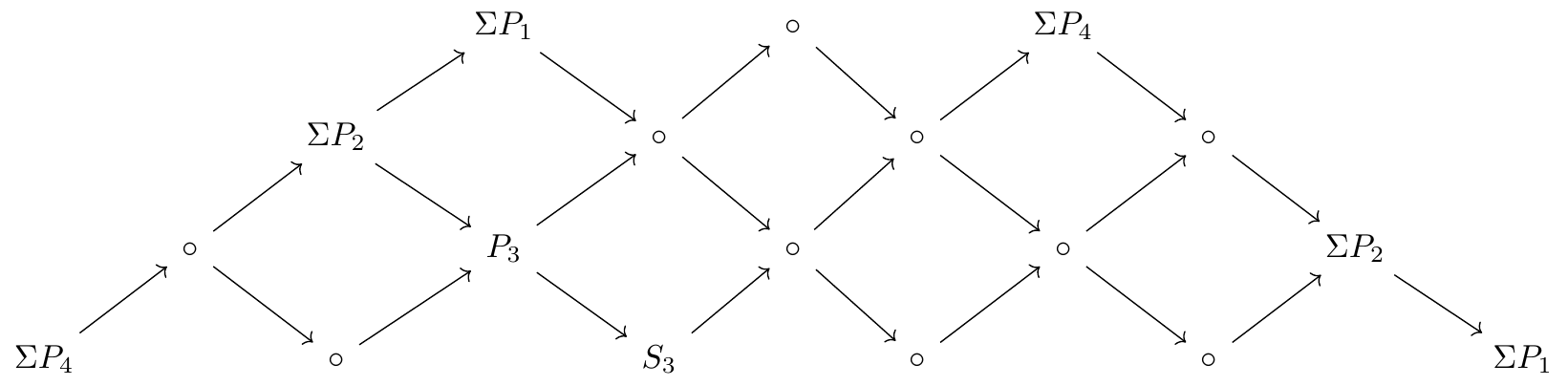}
\end{gather*}

\begin{gather*}
\add\sus R =\CS = \CV=\tensor[]{\XR}{^{\perp_1}} \\ 
\includegraphics[width=0.8\textwidth]{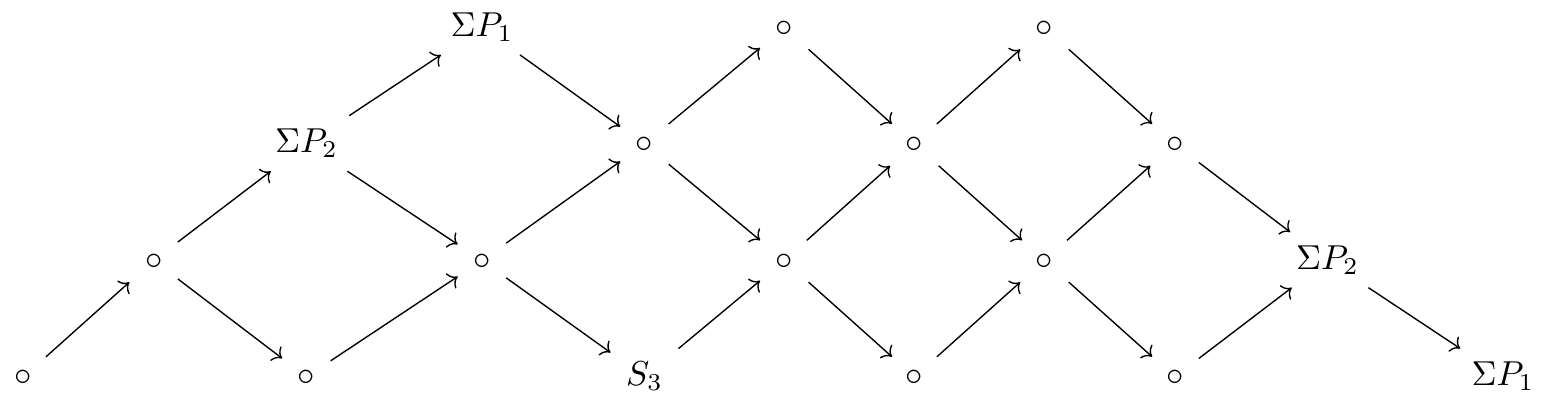}
\end{gather*}

By \cite[Prop. 2.9]{LiuS-AR-theory-in-KS-cat}, the quasi-abelian heart $\overline{\CH}=\CH/[\CW]=\CC/[\CX_{R}]$
for this twin cotorsion pair then has the following Auslander-Reiten
quiver (ignoring the objects denoted by a ``$\circ$'' that lie
in $\CX_{R}$ and again with the mesh relations omitted). 
$$\makebox[0.8\textwidth][l]{\includegraphics[width=0.8\textwidth]{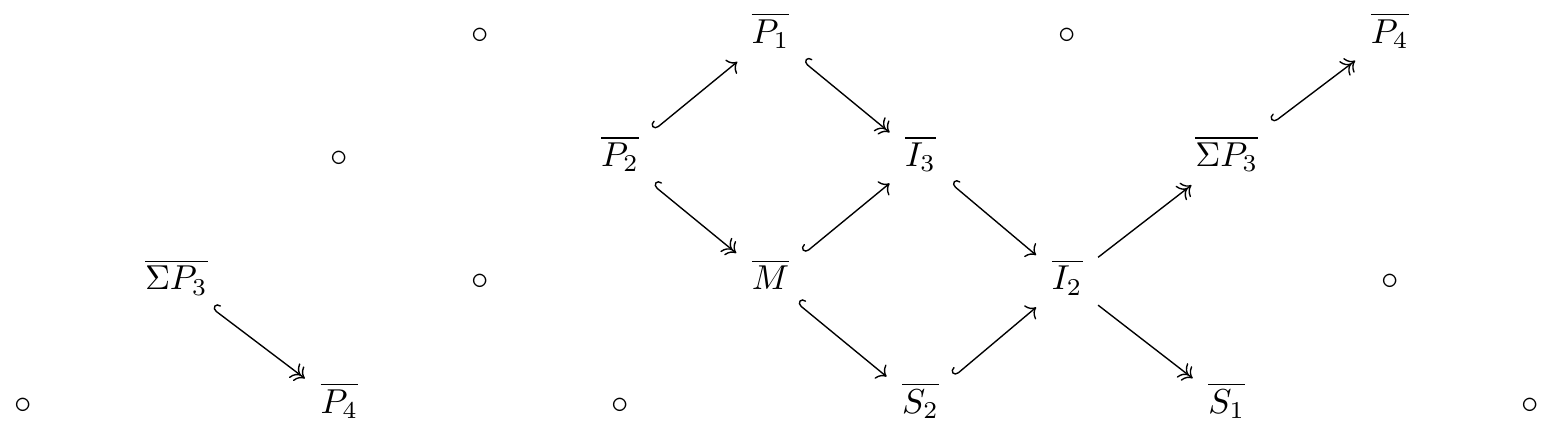}} $$where
one may define the Auslander-Reiten quiver for a Krull-Schmidt category
as in \cite{LiuS-AR-theory-in-KS-cat}. We have denoted
by $\overline{X}$ the image of the object $X$ of $\CC$ in $\CC/[\CX_{R}]$,
monomorphisms by ``$\hookrightarrow$'' and epimorphisms by ``$\twoheadrightarrow$''.
The extra righthand copy of $\overline{P_{4}}$ is included to illustrate
that this quiver really is connected and similar in shape to the Auslander-Reiten
quiver of $\mod\Lambda_{R}$ (see below). In this example there are
precisely three irreducible morphisms between indecomposables that are regular morphisms,
namely the morphisms $\overline{P_{1}}\to\overline{I_{3}}$, $\overline{P_{2}}\to\overline{M}$
and $\overline{\sus P_{3}}\to\overline{P_{4}}$. As noted in \S\ref{sec:Introduction},
one may show that various aspects of Auslander-Reiten theory are still
applicable in quasi-abelian categories. We refer the reader to \cite{Shah-AR-theory-quasi-abelian-cats-KS-cats}
for more details. However, one noticeable difference is that in a
quasi-abelian category there exist irreducible morphisms that are
regular. On the other hand, in an abelian category an irreducible
morphism cannot be regular, since a morphism is regular if and only
if it is an isomorphism in such a category, and irreducible morphisms
cannot be isomorphisms by definition.

In addition, one may obtain the Auslander-Reiten quiver of $\mod\Lambda_{R}$
by localising $\overline{\CH}$ at the regular morphisms as shown
in \cite{BuanMarsh-BM2}. In this case, one obtains the Auslander-Reiten 
quiver $$\makebox[0.3\textwidth][l]{\includegraphics[width=0.3\textwidth]{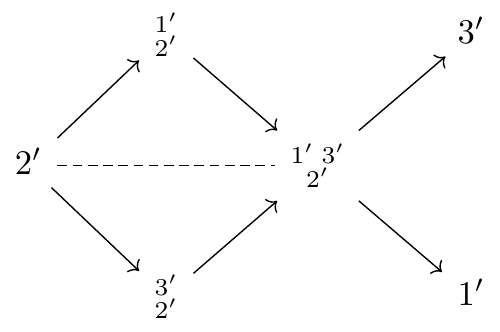}} $$where
$\Lambda_{R}$ is isomorphic to the path algebra of the quiver $1'\to 2'\leftarrow 3'$.
\end{example}
\begin{acknowledgements}
The author would like to thank Robert J. Marsh for his helpful guidance and support during the preparation of this article. The author is also grateful for financial support from the University of Leeds through a University of Leeds 110 Anniversary Research Scholarship. The author also thanks the referee for comments on an earlier version of the paper.
\end{acknowledgements}

\bibliography{mybib2}

\begin{thebibliography}{10}

\bibitem{Auslander-coherent-functors}
M.~Auslander.
\newblock {\em Coherent functors}.
\newblock In {\em Proc. {C}onf. {C}ategorical {A}lgebra ({L}a {J}olla,
  {C}alif., 1965)}, pages 189--231. Springer, New York, 1966.

\bibitem{AuslanderReiten-Rep-theory-of-Artin-algebras-IV}
M.~Auslander and I.~Reiten.
\newblock {\em Representation theory of {A}rtin algebras. {IV}. {I}nvariants
  given by almost split sequences}.
\newblock {\em Comm. Algebra}, {\bf 5}(5):443--518, 1977.

\bibitem{AuslanderReiten-Rep-theory-of-Artin-algebras-V}
M.~Auslander and I.~Reiten.
\newblock {\em Representation theory of {A}rtin algebras. {V}. {M}ethods for
  computing almost split sequences and irreducible morphisms}.
\newblock {\em Comm. Algebra}, {\bf 5}(5):519--554, 1977.

\bibitem{AuslanderReiten-apps-of-contravariantly-finite-subcats}
M.~Auslander and I.~Reiten.
\newblock {\em Applications of contravariantly finite subcategories}.
\newblock {\em Adv. Math.}, {\bf 86}(1):111--152, 1991.

\bibitem{AuslanderSmalo-preprojective-modules}
M.~Auslander and S.~O. Smal{\o{}}.
\newblock {\em Preprojective modules over {A}rtin algebras}.
\newblock {\em J. Algebra}, {\bf 66}(1):61--122, 1980.

\bibitem{Beligiannis-rigid-objects-triangulated-subfactors-and-abelian-localizations}
A.~Beligiannis.
\newblock {\em Rigid objects, triangulated subfactors and abelian
  localizations}.
\newblock {\em Math. Z.}, {\bf 274}(3-4):841--883, 2013.

\bibitem{BeilinsonBernsteinDeligne-perverse-sheaves}
A.~A. Be\u{\i}linson, J.~Bernstein, and P.~Deligne.
\newblock {\em Faisceaux pervers}.
\newblock In {\em Analysis and topology on singular spaces, {I} ({L}uminy,
  1981)}, volume 100 of {\em Ast\'{e}risque}, pages 5--171. Soc. Math. France,
  Paris, 1982.

\bibitem{BMRRT-cluster-combinatorics}
A.~B. Buan, R.~Marsh, M.~Reineke, I.~Reiten, and G.~Todorov.
\newblock {\em Tilting theory and cluster combinatorics}.
\newblock {\em Adv. Math.}, {\bf 204}(2):572--618, 2006.

\bibitem{BuanMarsh-BM2}
A.~B. Buan and R.~J. Marsh.
\newblock {\em From triangulated categories to module categories via
  localization {II}: calculus of fractions}.
\newblock {\em J. Lond. Math. Soc. (2)}, {\bf 86}(1):152--170, 2012.

\bibitem{BuanMarsh-BM1}
A.~B. Buan and R.~J. Marsh.
\newblock {\em From triangulated categories to module categories via
  localisation}.
\newblock {\em Trans. Amer. Math. Soc.}, {\bf 365}(6):2845--2861, 2013.

\bibitem{BuanMarshReiten-cluster-tilted-algebras}
A.~B. Buan, R.~J. Marsh, and I.~Reiten.
\newblock {\em Cluster-tilted algebras}.
\newblock {\em Trans. Amer. Math. Soc.}, {\bf 359}(1):323--332, 2007.

\bibitem{BucurDeleanu-intro-to-theory-of-cats-and-functors}
I.~Bucur and A.~Deleanu.
\newblock {\em Introduction to the theory of categories and functors}.
\newblock With the collaboration of Peter J. Hilton and Nicolae Popescu. Pure
  and Applied Mathematics, Vol. XIX. Interscience Publication John Wiley \&
  Sons, Ltd., London-New York-Sydney, 1968.

\bibitem{CalderoChapotonSchiffler-quivers-arising-from-cluster-a_n-case}
P.~Caldero, F.~Chapoton, and R.~Schiffler.
\newblock {\em Quivers with relations arising from clusters ({$A_n$} case)}.
\newblock {\em Trans. Amer. Math. Soc.}, {\bf 358}(3):1347--1364, 2006.

\bibitem{Dickson-torsion-theory-for-abelian-cats}
S.~E. Dickson.
\newblock {\em A torsion theory for {A}belian categories}.
\newblock {\em Trans. Amer. Math. Soc.}, {\bf 121}:223--235, 1966.

\bibitem{Enochs-inj-flat-covers-envelopes-and-resolvents}
E.~E. Enochs.
\newblock {\em Injective and flat covers, envelopes and resolvents}.
\newblock {\em Israel J. Math.}, {\bf 39}(3):189--209, 1981.

\bibitem{GabrielZisman-calc-of-fractions}
P.~Gabriel and M.~Zisman.
\newblock {\em Calculus of fractions and homotopy theory}.
\newblock Ergebnisse der Mathematik und ihrer Grenzgebiete, Band 35.
  Springer-Verlag New York, Inc., New York, 1967.

\bibitem{Happel-triangulated-cats-in-rep-theory}
D.~Happel.
\newblock {\em Triangulated categories in the representation theory of
  finite-dimensional algebras}, volume 119 of {\em London Mathematical Society
  Lecture Note Series}.
\newblock Cambridge University Press, Cambridge, 1988.

\bibitem{HolmJorgensen-tri-cats-intro}
T.~Holm and P.~J{\o{}}rgensen.
\newblock {\em Triangulated categories: definitions, properties, and examples}.
\newblock In {\em Triangulated categories}, volume 375 of {\em London Math.
  Soc. Lecture Note Ser.}, pages 1--51. Cambridge Univ. Press, Cambridge, 2010.

\bibitem{Iyama-higher-dimnl-AR-theory-maximal-orthog-subcats}
O.~Iyama.
\newblock {\em Higher-dimensional {A}uslander-{R}eiten theory on maximal
  orthogonal subcategories}.
\newblock {\em Adv. Math.}, {\bf 210}(1):22--50, 2007.

\bibitem{IyamaYoshino-mutation-in-tri-cats-rigid-CM-mods}
O.~Iyama and Y.~Yoshino.
\newblock {\em Mutation in triangulated categories and rigid {C}ohen-{M}acaulay
  modules}.
\newblock {\em Invent. Math.}, {\bf 172}(1):117--168, 2008.

\bibitem{Jorgensen-auslander-reiten-triangles-in-subcategories}
P.~J{\o{}}rgensen.
\newblock {\em Auslander-{R}eiten triangles in subcategories}.
\newblock {\em J. K-Theory}, {\bf 3}(3):583--601, 2009.

\bibitem{Keller-cy-tri-cats}
B.~Keller.
\newblock {\em Calabi-{Y}au triangulated categories}.
\newblock In {\em Trends in representation theory of algebras and related
  topics}, EMS Ser. Congr. Rep., pages 467--489. Eur. Math. Soc., Z{\"{u}}rich,
  2008.

\bibitem{KellerReiten-ct-algebras-are-gorenstein-stably-cy}
B.~Keller and I.~Reiten.
\newblock {\em Cluster-tilted algebras are {G}orenstein and stably
  {C}alabi-{Y}au}.
\newblock {\em Adv. Math.}, {\bf 211}(1):123--151, 2007.

\bibitem{KoenigZhu-from-tri-cats-to-abelian-cats}
S.~Koenig and B.~Zhu.
\newblock {\em From triangulated categories to abelian categories: cluster
  tilting in a general framework}.
\newblock {\em Math. Z.}, {\bf 258}(1):143--160, 2008.

\bibitem{KrauseSaorin-minimal-approximations-of-modules}
H.~Krause and M.~Saor\'{i}n.
\newblock {\em On minimal approximations of modules}.
\newblock In {\em Trends in the representation theory of finite-dimensional
  algebras ({S}eattle, {WA}, 1997)}, volume 229 of {\em Contemp. Math.}, pages
  227--236. Amer. Math. Soc., Providence, RI, 1998.

\bibitem{LiuS-AR-theory-in-KS-cat}
S.~Liu.
\newblock {\em Auslander-{R}eiten theory in a {K}rull-{S}chmidt category}.
\newblock {\em S{\~{a}}o Paulo J. Math. Sci.}, {\bf 4}(3):425--472, 2010.

\bibitem{LiuYNakaoka-hearts-of-twin-cotorsion-pairs-on-extriangulated-categories}
Y.~Liu and H.~Nakaoka.
\newblock {\em Hearts of twin cotorsion pairs on extriangulated categories}.
\newblock {\em J. Algebra}, {\bf 528}:96--149, 2019.

\bibitem{MarshPalu-nearly-morita-equivalences-and-rigid-objects}
R.~J. Marsh and Y.~Palu.
\newblock {\em Nearly {M}orita equivalences and rigid objects}.
\newblock {\em Nagoya Math. J.}, {\bf 225}:64--99, 2017.

\bibitem{Nakaoka-cotorsion-pairs-I}
H.~Nakaoka.
\newblock {\em General heart construction on a triangulated category ({I}):
  {U}nifying {$t$}-structures and cluster tilting subcategories}.
\newblock {\em Appl. Categ. Structures}, {\bf 19}(6):879--899, 2011.

\bibitem{Nakaoka-twin-cotorsion-pairs}
H.~Nakaoka.
\newblock {\em General heart construction for twin torsion pairs on
  triangulated categories}.
\newblock {\em J. Algebra}, {\bf 374}:195--215, 2013.

\bibitem{Popescu-abelian-cats-with-apps-to-rings-and-modules}
N.~Popescu.
\newblock {\em Abelian categories with applications to rings and modules}.
\newblock Academic Press, London-New York, 1973.
\newblock London Mathematical Society Monographs, No. 3.

\bibitem{Rump-almost-abelian-cats}
W.~Rump.
\newblock {\em Almost abelian categories}.
\newblock {\em Cahiers Topologie G\'{e}om. Diff\'{e}rentielle Cat\'{e}g.}, {\bf
  42}(3):163--225, 2001.

\bibitem{Rump-counterexample-to-Raikov}
W.~Rump.
\newblock {\em A counterexample to {R}aikov's conjecture}.
\newblock {\em Bull. Lond. Math. Soc.}, {\bf 40}(6):985--994, 2008.

\bibitem{Salce-cotorsion-theories-for-abelian-groups}
L.~Salce.
\newblock {\em Cotorsion theories for abelian groups}.
\newblock In {\em Symposia {M}athematica, {V}ol. {XXIII} ({C}onf. {A}belian
  {G}roups and their {R}elationship to the {T}heory of {M}odules, {INDAM},
  {R}ome, 1977)}, pages 11--32. Academic Press, London-New York, 1979.

\bibitem{Schiffler-quiver-reps-book}
R.~Schiffler.
\newblock {\em Quiver representations}.
\newblock CMS Books in Mathematics/Ouvrages de Math{\'{e}}matiques de la SMC.
  Springer, Cham, 2014.

\bibitem{Shah-AR-theory-quasi-abelian-cats-KS-cats}
A.~Shah.
\newblock {\em Auslander-{R}eiten theory in quasi-abelian and {K}rull-{S}chmidt
  categories}.
\newblock To appear in J. Pure Appl. Algebra, 2019.
  \url{https://doi.org/10.1016/j.jpaa.2019.04.017}.

\end{thebibliography}
\bibliographystyle{mybst}

\end{document}